\documentclass[12pt,plain]{article}
\usepackage{fancyhdr}
\usepackage{verbatim}
\usepackage{indentfirst}
\usepackage{graphicx}
\usepackage{epstopdf}
\usepackage{color}
\usepackage{newlfont}
\usepackage{amssymb}
\usepackage[intlimits]{amsmath}
\usepackage{latexsym}
\usepackage{amsthm}
\usepackage{amscd}
\usepackage{fullpage}
\usepackage{setspace}
\usepackage{multirow}
\usepackage{hyperref}
\usepackage[dvips, hmargin=2.1cm,vmargin=3cm]{geometry}

\newcommand{\de}{\mathrm{d}}
\newcommand{\virg}[1]{``#1''}
\newcommand{\N}{\mathbb{N}}
\newcommand{\Z}{\mathbb{Z}}

\newcommand{\C}{\mathbb{C}}

\newcommand{\bianco}{\textcolor{white}{.}}

\newcommand{\be}{\begin{equation}}
\newcommand{\bey}{\begin{eqnarray}}
\newcommand{\ee}{\end{equation}}
\newcommand{\eey}{\end{eqnarray}}
\newcommand{\ba}{\begin{array}}
\newcommand{\ea}{\end{array}}

\newcommand{\bm}[1]{\mbox{\boldmath{$#1$}}}

\theoremstyle{plain}                    
\newtheorem{theorem}{Theorem}[section]
\newtheorem{lem}[theorem]{Lemma}            
\newtheorem{prop}[theorem]{Proposition}
\newtheorem{cor}[theorem]{Corollary}
\theoremstyle{definition}
\newtheorem{defin}[theorem]{Definition}
           
\newtheorem{remark}[theorem]{Remark}
\newtheorem*{maintheorem}{Main Theorem}
\newtheorem*{maincor}{Corollary}

\begin{document}
\title{Ergodic Properties of Square-Free Numbers}
\author{F. Cellarosi$^*$, Ya.G. Sinai$^{\S,\dag}$}
\maketitle

\let\oldthefootnote\thefootnote
\renewcommand{\thefootnote}{\fnsymbol{footnote}}
\footnotetext[1]{Institute for Advanced Study. Princeton, NJ, U.S.A.}
\footnotetext[2]{Mathematics Department, Princeton University. Princeton, NJ, U.S.A.}
\footnotetext[3]{Landau Institute of Theoretical Physics, Russian Academy of Sciences. Moscow, Russia.}
\let\thefootnote\oldthefootnote

\begin{abstract}We construct a natural invariant measure concentrated on the set of square-free numbers, and invariant under the shift. We prove that the corresponding dynamical system is isomorphic to a translation on a compact, Abelian group. This implies that this system is not weakly mixing and has zero measure-theoretical entropy.
\end{abstract}

\begin{keywords} square-free numbers, correlation functions, dynamical systems with pure point spectrum, ergodicity. {\bf MSC}: 37A35, 37A45, 28D99.
\end{keywords}


%

\section*{Introduction and Notations}
Let $\mathcal P$ be the set of prime numbers. By $p$ (with or without indices) we will always denote an element of $\mathcal P$. A positive integer $n$ is \emph{square-free} if $p^2\nmid n$ for every $p$. Denote the set of all square-free numbers by $\mathcal Q$ (for \emph{quadratfrei}). The indicator of the set $\mathcal Q$ is the function $n\mapsto\mu^2(n)$, where $\mu$ is the M\"{o}bius function: 
\be\mu(n)=\begin{cases}1,&\mbox{if $n=1$};\\ 0,&\mbox{if $n$ is not square-free;}\\(-1)^k,&\mbox{if $n$ is the product of $k$ distinct primes.}\end{cases}\nonumber\ee
The functions $\mu$ and $\mu^2$ are of great importance in Number Theory because of their connection with the Riemann zeta function. For example,
\be
\sum_{n=1}^\infty\frac{\mu(n)}{n^s}=\frac{1}{\zeta(s)},\hspace{.5cm}\sum_{n=1}^\infty\frac{\mu^2(n)}{n^s}=\frac{\zeta(s)}{\zeta(2s)}.\nonumber 
\ee

Furthermore, the estimate $\left|\sum_{n\leq N}\mu(n)\right|=\mathcal O_\varepsilon(N^{1/2+\varepsilon})$ as $N\to\infty$ is equivalent to the Riemann Hypothesis. P. Sarnak \cite{Sarnak-Mobius-lectures} has recently addressed a number of 
statistical and ergodic properties of the sequences $(\mu(n))_n$ and $(\mu^2(n))_n$. 

\subsection{Notations}
We shall use the standard notation $e(x)=e^{2\pi i x}$.
For every integer $n$ denote by $\omega(n)$ the number of its distinct prime factors. For example, $\omega(1)=0$ and $\omega(2\cdot 3)=\omega(2^{10}\cdot 3^7)=\omega(7\cdot 23)=2$.
We shall also use the notations
\bey
\mathcal P(n)=\{p:\:p|n\},\hspace{.5cm}\mathcal P_2(n)=\{p:\:p^2|n\}.\nonumber
\eey
Notice that if $n\in\mathcal Q$, then $|\mathcal P(n)|=\omega(n)$, $\mathcal P_2(n)=\varnothing$, and $\mathcal P_2(n^2)=\mathcal P(n)$.
For every finite set $\mathcal A\subset\mathcal P$, define 
\be[\mathcal A]=\prod_{p\in\mathcal A}p\nonumber. \ee 
In particular $[\varnothing]=1$. Notice that if $\mathcal A,\mathcal B$ are disjoint, then $[\mathcal A\cup\mathcal B]=[\mathcal A][\mathcal B]$ and $[\mathcal A\cap\mathcal B]=1$.

\section{Formulation of the results}\label{sec1}
The goal of this paper is  to describe a dynamical system `naturally' associated to $\mathcal Q$ and study its statistical and ergodic properties.

\subsection{Correlation functions}\label{subsec-corr-fun}
The first step is the construction of \emph{correlation functions} for $\mathcal Q$. 
Choose $r$ integers $0\leq k_1<k_2<\ldots<k_r$ and consider the set
\be\mathcal Q_N(k_1,k_2,\ldots,k_r)=\{n\leq N:\: \mu^2(n)=\mu^2(n+k_1)=\ldots=\mu^2(n+k_r)=1\}.\nonumber\ee
The ratio 
\be \mathbb E_N(k_1,k_2,\ldots,k_r):=\frac{|\mathcal Q_N(k_1,k_2,\ldots,k_r)|}{N}\label{def-E_N(k1-kr)}\ee
is the frequency of square-free integers $n\leq N$ for which $n+k_1, n+k_2, \ldots, n+k_r$ are also square-free. It also gives the expectation (hence the notation $\mathbb E$) of the product $\mu^2(n)\mu^2(n+k_1)\cdots\mu^2(n+k_r)$ with respect to the uniform measure on $\{1,2,\ldots,N\}$. Notice, by taking $r=1$ and $k_1=0$, that $\mathcal Q_N(0)$ is simply the set of all square-free numbers not greater than $N$. It is well known that 
\be\lim_{N\to\infty}\mathbb E_N(0)=\frac{6}{\pi^2}\approx 0.6079271018\label{density-square-free}\ee
We include the proof of (\ref{density-square-free}) and some of its generalizations in Section \ref{sec2}, see Theorems \ref{thm1}-\ref{thm1''}.
The study of $\mathbb E_N(k_1,\ldots, k_r)$ as $N\to\infty$ is also classical, see L. Mirsky \cite{Mirsky-1949}, R.R. Hall \cite{Hall-1989}, K.M. Tsang \cite{Tsang-1986}, D.R. Heath-Brown \cite{Heath-Brown-1984}.
It is known that the limits 
\bey
c_{r+1}(k_1,\ldots,k_r)=\lim_{N\to\infty}\mathbb E_N(k_1,\ldots,k_r)\label{limit-c_r-quote}
\eey
exist. 
We shall refer to $c_{r+1}$ as the \emph{$(r+1)$-st correlation function for $\mathcal Q$}. Various formul\ae\: for
$c_{r+1}(k_1,\ldots,k_r)$ are known (see Section \ref{sec5}). We shall rewrite the one by L. Mirsky \cite{Mirsky-1949} to express the correlation functions as a sum, namely 
\bey c_{r+1}(k_1,\ldots,k_r)=\sum_{0\leq l'<l''\leq r}\sum_{\scriptsize{\ba{c}\mu^2(d_{l',l''})=1\\ d_{l',l''}^2|k_{l'}-k_{l''}\ea}}\sum_{m_0,m_1,\ldots,m_r\geq 0}(-1)^{\sum_{l=0}^r m_l}\sum_{\scriptsize{\ba{c}\mathcal P_l\subset\mathcal P,\: 0\leq l\leq r \\ |\mathcal P_l|=m_l\\\mbox{$[\mathcal P_{l'}\cap\mathcal P_{l''}]=d_{l',l''}$}\ea}}\frac{1}{\left[\bigcup_{l=0}^r \mathcal P_l\right]^2}.\label{general-c_r+1-sinai-toquote}\eey
The above formula, 
although complicated, plays a role in the spectral analysis of the correlation functions. Let, for example, $r=1$. For every $d\in\mathcal Q$  define
\be\sigma_d= \displaystyle{\sum_{m_0,m_1\geq 0}}(-1)^{m_0+m_1}\sum_{\scriptsize{\ba{c}\mathcal P_0,\mathcal P_1\subset \mathcal P\\|\mathcal P_0|=m_0,\,|\mathcal P_1|=m_1\\ \mbox{$[\mathcal P_0\cap\mathcal P_1]=d$}\ea}}\frac{1}{\left[\mathcal P_0\cup \mathcal P_1\right]^2}.\label{def-sigma_d}\ee
Explicit formul\ae\, for $\sigma_d$ are given in Section \ref{sec5}.
Then 
\be c_2(k)=\sum_{\scriptsize{\ba{c}\mu^2(d)=1\\ d^2|k\ea}}\sigma_d\label{statement-thm3-toquote}\ee
and the corresponding spectral 
measure $\nu$ on $\mathbb S^1$ (i.e. satisfying $c_2(k)=\hat\nu(k)$ by Bochner theorem) 
is pure point, given as sum of $\delta$-functions at the points $e(l/d^2)$, where $d\in\mathcal Q$. More precisely,
\be\nu=\sum_{\mu^2(d)=1}\sigma_d\sum_{l=0}^{d^2-1}\delta_{e(l/d^2)},\label{spectral-measure-nu}\ee
where the convergence of the series is guaranteed by Lemma \ref{prop-formula-sigma-d}. 
The spectrum (i.e. the support of $\nu$) is the group 
\be\Lambda=\left\{e\!\left(\frac{l}{d^2}\right)\!:\:\: 0\leq l\leq d^2-1,\:\:\mu^2(d)=\mu^2(\gcd(l,d^2))=1\right\}\subset\mathbb S^1.\nonumber 
\ee
Notice that 
every element of $\Lambda$ is represented uniquely. 
Moreover, every rational number of the form $\frac{l}{d^2}$ such that $d$ is square-free, $0\leq l\leq d^2-1$, and $\gcd(l,d^2)$ is also square-free can be written as
\be
\frac{l}{d^ 2}=\frac{l_1}{p_1^2}+\ldots+\frac{l_r}{p_m^2}\label{representation-l/d^2},
\ee
for some $l_1,\ldots,l_m$, where $\{p_1,\ldots,p_m\}=\mathcal P(d)$.
This representation (\ref{representation-l/d^2}) is unique if one imposes the restriction $0\leq l_j\leq p_m^2-1$, $1\leq j\leq m$. In other words, the group $\Lambda$ 
is isomorphic to the direct sum $\bigoplus_p\Z/p^2\Z$ (where only finitely many coordinates are non-zero). Therefore, $\Lambda$ is the Pontryagin dual of the direct product group \be\mathbb G
=\prod_p\Z/p^2\Z,\label{groupG}\ee
which is an Abelian compact group (endowed with the product topology). In other words, $\hat {\mathbb G
}\cong\Lambda$.
Each element $\mathbf{g}\in\mathbb G$ is identified with a sequence $(g_{p^2})_{p\in\mathcal P}$ indexed by $\mathcal P$, where $g_{p^2}\in\Z/p^2\Z$: 
\be\mathbf{g}\equiv(g_4,g_9,g_{25},g_{49},\ldots).\nonumber
\ee
Given $\mathbf h\in\mathbb G$, denote by $\mathbb T_\mathbf{h}:\mathbb G\to\mathbb G$ the translation $\mathbb T_{\mathbf h}(\mathbf g)=\mathbf g+\mathbf h$.
Let $\mathbb B$ be the natural $\sigma$-algebra on $\mathbb G$, and let us put the uniform 
measure on each $\Z/p^2\Z$. The corresponding product measure  $\mathbb P$ on $\mathbb B$ is invariant under translations, and therefore it is the Haar measure. 

The ergodic properties of translations on compact Abelian groups (also known as \emph{Kronecker systems}) were studied for the first time by J. von Neumann. He showed \cite{vonNeumann-1932} that such two ergodic translations with the same spectrum, are isomorphic as measure-preserving dynamical systems. This is true in general for ergodic transformations with pure point spectrum and it 
plays an important role in our analysis. Later, P.R. Halmos and J. von Neumann \cite{Halmos-vonNeumann-1942} proved that every ergodic dynamical system with pure point spectrum is isomorphic to a translation on a compact Abelian group. This implies, for example, that every ergodic dynamical system with pure point spectrum is isomorphic to its inverse. For an historical survey on the isomorphism problem see \cite{pittphilsci8813}.

\subsection{A Natural Dynamical System}\label{natural-dyn-sys}
Consider the space $X$ of all bi-infinite sequences $x=\{x(n),\:-\infty<n<\infty\}$ where each $x(n)$ takes value either 0 or 1. Denote by 
$\mathcal B$ the natural $\sigma$-algebra generated by cylinder sets, 
and introduce the probability measure $\Pi$ defined on 
$\mathcal B$ as follows: For every $r\geq 0$ and every $-\infty<k_0<k_1<\ldots<k_r<\infty$
\be\Pi\left\{ x\in X:\: x(k_0)=x(k_1)=\ldots=x(k_r)=1\right\}=c_{r+1}(k_1-k_0,k_2-k_0,\ldots, k_r-k_0)\label{definition-measure-Pi-toquote},\ee
where $c_{r+1}$ is the $(r+1)$-st correlation function (\ref{general-c_r+1-sinai-toquote}). 
It is clear that (\ref{definition-measure-Pi-toquote}) determines the measure $\Pi$ uniquely. We call $\Pi$ \emph{the natural measure corresponding to the set of square-free numbers}. 

If $T$ is the shift on $X$, i.e. $Tx=x'$, $x'(n)=x(n+1)$, then it follows immediately from (\ref{definition-measure-Pi-toquote}) that $\Pi$ is invariant under $T$. 
We can now formulate the main result of this paper:
\begin{maintheorem}\bianco 
\begin{itemize}\item[(i)]
The dynamical system $(X,\mathcal B, \Pi,T)$ is ergodic and has pure point spectrum given by $\Lambda$. 
\item[(ii)] $(X,\mathcal B, \Pi,T)$ is isomorphic to $(\mathbb G, \mathbb B,\mathbb P,\mathbb T_{\mathbf u})$, where $\mathbf u=(1,1,1,\ldots)$.
\end{itemize} 
\end{maintheorem}

P. Sarnak \cite{Sarnak-Mobius-lectures} formulates the result that $(\mathbb G, \mathbb B,\mathbb P,\mathbb T_{\mathbf u})$ is a factor of $(X,\mathcal B, \Pi,T)$. His methods also allow to show that the factor map is in fact an isomorphism. Our approach is rather different and  is based on a spectral analysis of the dynamical system $(X, \mathcal B, \Pi, T)$. The statement in the following corollary can be also found in \cite{Sarnak-Mobius-lectures}.
\begin{maincor}
The dynamical system $(X,\mathcal B, \Pi,T)$ is non weakly-mixing, and its measure-theoretic entropy is zero. 
\end{maincor}

It is worthwhile to remark that the main focus of \cite{Sarnak-Mobius-lectures} are the topological dynamical systems $M=(O_{\mu(n)},T)$, $S=(O_{\mu^2(n)},T)$ given by the shifts on the orbit closures of $(\mu(n))_{n}$ and $(\mu^2(n))_n$, respectively. The topological entropy of $S$ is positive, equal to $\frac{6}{\pi^2}\log 2$. R. Peckner \cite{Peckner-2012} recently constructed a measure of maximal entropy for
$S$; he showed that this measure is unique, and the corresponding dynamical system is isomorphic to the direct product of 
$(\mathbb G,\mathbb B,\mathbb P,\mathbb T_{\mathbf u})$ and a Bernoulli shift with entropy $\frac{6}{\pi^2}\log 2$. In particular, the dynamical system $(X,\mathcal B, \Pi, T)$ that we consider is its the Pinsker factor.

Our paper is organized as follows. 
Section \ref{sec2} includes the classical computation of the density of square-free numbers and its generalization to square-free numbers avoiding finite sets of prime factors (the proof is given in Appendix \ref{appendix-A}). The latter will be used for the computation of certain relevant constants.
Section \ref{sec5} contains various formul\ae\, for the correlation functions, including the derivation of (\ref{statement-thm3-toquote}) and (\ref{general-c_r+1-sinai-toquote}) from the formula by L. Mirsky.
Section \ref{sec-averages-corr-fun} includes several useful lemmata (some of which are proven in Appendix \ref{appendix-B}) concerning averages and exponential sums for the correlation functions. These results are crucial for the spectral analysis of the dynamical system  $(X,\mathcal B, \Pi,T)$.
%
%
Such 
analysis 
is carried out in Section \ref{sec-spectrum-of-T} and yields the first part of our Main Theorem. 
The analysis of the spectrum for $(\mathbb G,\mathbb B,\mathbb P,\mathbb T_{\mathbf u})$ is done in Section \ref{sec-spectrum-of-G}, and the second part of our Main theorem follows from it, by means of 
a theorem by J. von Neumann \cite{vonNeumann-1932}.

\section*{Acknowledgments}
The authors thank M. Boyle, M. Degli Esposti, G. Forni, P. Sarnak, I. Shkredov, I. Vinogradov for useful discussions, and the anonymous referees for their suggestions to improve on the first version of this paper. 
The first author's work is 
supported by the National Science Foundation under agreement No. DMS-0635607. 
The second author acknowledges the financial support from the NSF, grant No. DMS-0901235.

\section{The density of $\mathcal Q$ and some of its subsets}\label{sec2}
Recall that $\mathbb E_N(0)=\frac{1}{N}\left|\{n\leq N,\: n\in\mathcal Q\}\right|$. The following theorem is very classical.
\begin{theorem}\label{thm1}
\be\lim_{N\to\infty}\mathbb E_N(0)=\frac{6}{\pi^2}\label{density-square-free-new}\ee
\end{theorem}

\begin{proof}
We can write $\mu^2$ as the indicator of the set of square-free numbers by imposing the condition that its argument avoids all arithmetic progressions modulo $p^2$:
\be\mu^2(n)=\prod_{p}\left(1-\chi_{p^2}(n)\right)\label{mu^2=prod}\ee
In the above expression $\chi_{p^2}(n)$ is the indicator of the arithmetic progression $\{p^2l:\:l\in\Z\}$. Let us open the brackets in (\ref{mu^2=prod}):
\be
\mu^2(n)=1-\sum_{p}\chi_{p^2}(n)+\sum_{p_1<p_2}\chi_{p_1^2}(n)\chi_{p_2^2}(n)-\sum_{p_1<p_2<p_3}\chi_{p_1^2}(n)\chi_{p_2^2}(n)\chi_{p_3^2}(n)+\ldots.\nonumber
\ee
We can write
\bey
\mathbb E_N(0) 
&=&\frac{1}{N}\sum_{n\leq N}\mu^2(n)=1-\sum_{p}\frac{1}{p^2}+\sum_{p_1<p_2}\frac{1}{(p_1p_2)^2}-\sum_{p_1<p_2<p_3}\frac{1}{(p_1p_2p_3)^2}+\ldots+\varepsilon_N=\nonumber\\
&=&\prod_{p}\left(1-\frac{1}{p^2}\right)+\varepsilon_N=\frac{1}{\zeta(2)}+\varepsilon_N=\frac{6}{\pi^2}+\varepsilon_N.\nonumber
\eey
Here and below $\varepsilon_N$ 
denotes a remainder that tends to zero as $N\to\infty$.
\end{proof}
The statement of Theorem 1 can actually be refined as follows:
\begin{theorem}\label{thm1'}
\be \mathbb E_N(0)=\frac{6}{\pi^2}+\mathcal O(N^{-1/2})\hspace{.5cm}\mbox{as $N\to\infty$.}\nonumber
\ee
\end{theorem}
In other words, $\varepsilon_N$ in the proof of Theorem \ref{thm1} satisfies the estimate $|\varepsilon_N|=\mathcal O(N^{-1/2})$.
This result is also very classical, and 
is a special case of Theorem \ref{thm1''} below.
Let us fix a finite set $\mathcal{S}\subset \mathcal P$ and define the set
\be\mathcal Q_N^{\mathcal S}(0)=\{n\leq N:\: \mu^2(n)=1,\: p\in\mathcal S\Rightarrow p\nmid n\}\label{def-Q_N^S(0)}\ee
of all square-free numbers not bigger than $N$ and not divisible by any of the primes $p\in\mathcal S$. For example, $\mathcal Q_N^{\{2\}}(0)$ is the set of odd square-free numbers not bigger than $N$. Notice that when $\mathcal S$ is empty we get the full set of square-free numbers, i.e. $\mathcal Q_N^{\varnothing}(0)=\mathcal Q_N(0)$.
In analogy with (\ref{def-E_N(k1-kr)}), let us define
\be\mathbb E_N^{\mathcal S}(0)=\frac{|\mathcal Q_N^{\mathcal S}(0)|}{N}\nonumber\ee
We have the following
\begin{theorem}\label{thm1''} For every finite $\mathcal S\subset \mathcal P$ we have
\be\mathbb E_N^{\mathcal S}(0)=\frac{\alpha(\mathcal S)}{\zeta(2)}+\mathcal O_{\mathcal S}(N^{-1/2})\hspace{.5cm}\mbox{as $N\to\infty$},\nonumber\ee
where \be\alpha(\mathcal S)=\prod_{p\in\mathcal S}\frac{p}{p+1}\nonumber\ee
and the constant $C(\mathcal S)$ implied by the $\mathcal O_{\mathcal S}$-notation can be taken as
 \be C(\mathcal S)=4\prod_{p\in\mathcal S}\frac{p-1}{p}+\left(\prod_{p\in\mathcal S}p-1\right)-\prod_{p\in\mathcal S}(p-1).\nonumber\ee
\end{theorem}
The proof of Theorem \ref{thm1''} is presented in Appendix \ref{appendix-A}; it 
implies the existence of the asymptotic densities
\be \lim_{N\to\infty}\mathbb E_N^{\mathcal S}(0)=\frac{\alpha(\mathcal S)}{\zeta(2)}.\label{densities-E_N^S(0)}\ee
For example, the density of the set of odd square-free numbers is $\frac{4}{\pi^2}$ (i.e. odd and even square-free numbers are in 2:1 proportion). Analogously, by choosing $\mathcal S=\{p\}$,  we see that the set of square-free numbers not divisible by $p$ is \virg{$p$ times as large} (in the sense of density) as the set of those divisible by $p$.
If, for instance, we choose $\mathcal S=\{2,3\}$ we obtain $\alpha(\{2,3\})=1/2$, and we see that 50\% of the square-free numbers is not divisible by either 2 or 3.

\section{The Formul\ae\, for the correlation functions}\label{sec5}
L. Mirsky \cite{Mirsky-1949} proved that
\be c_{r+1}(k_1,k_2,\ldots,k_r)=\prod_{p}\left(1-\frac{A_p^{(r+1)}(k_1,k_2,\ldots,k_r)}{p^2}\right),\label{general-formula-c_r+1-Mirsky}\ee
where $A_p^{(r+1)}(k_1,k_2,\ldots,k_r)=|\{0,k_1 (\!\!\!\!\mod p^2), k_2 (\!\!\!\!\mod p^2),\ldots, k_r (\!\!\!\!\mod p^2)\}|$. 
Notice that \be 1\leq A_p^{(r+1)}(k_1, k_2, \ldots, k_r)\leq r \nonumber\ee 
for finitely many $p$ and $A_p^{(r+1)}(k_1,\ldots,k_r)=r+1$ for infinitely many $p$. 
For $r=1$, we have 
\be A_p^{(2)}(k)=\begin{cases}1& p^2| k;\\2,&\mbox{otherwise}.\end{cases}\nonumber\ee
This gives, for instance,
\be c_2(k)=\prod_{p^2|k}\left(1-\frac{1}{p^2}\right)\prod_{p^2\nmid k}\left(1-\frac{2}{p^2}\right)\label{Mirsky-c2}.\ee
It will be useful for us to write $c_2(k)$ (and in general $c_{r+1}(k_1,\ldots,k_r)$) as a sum. 
Recall the definition 
of $\sigma_d$ from Section \ref{subsec-corr-fun}. We prove the following formula for $\sigma_d$:
\begin{lem}\label{prop-formula-sigma-d}
\be\sigma_d=\frac{1}{d^2}\prod_{p\,\nmid d}\left(1-\frac{2}{p^2}\right).\label{formula-sigma_d}\ee
\end{lem}
\begin{proof}
Recall that, since $d$ is square-free, $|\mathcal P(d)|=\omega(d)$. By setting $m=m_1-\omega(d)$ and $M=m_1+m_2-2\omega(d)$ in (\ref{def-sigma_d}) we obtain
\bey
\sigma_d&=&\sum_{0\leq m\leq M}(-1)^{-2\omega(d)}(-1)^M{M\choose m}\sum_{\scriptsize{\ba{c}\mathcal P'\subset\mathcal P\smallsetminus\mathcal P(d) \\ | \mathcal P'|=M\ea}}\frac{1}{d^2}\frac{1}{[\mathcal P']^2}=\nonumber\\
&=&\frac{1}{d^2}\sum_{M\geq 0}\sum_{\scriptsize{\ba{c}\mathcal P'\subset\mathcal P\smallsetminus\mathcal P(d) \\ | \mathcal P'|=M\ea}}\frac{(-2)^M}{[\mathcal P']^2}=\frac{1}{d^2}\prod_{p\nmid d}\left(1-\frac{2}{p^2}\right).\nonumber\eey
\end{proof}
In particular, Lemma \ref{prop-formula-sigma-d} shows that $\sigma_d$ is positive and bounded away from zero and infinity. More precisely 
\be0<\sigma_1\leq \sigma_d<\frac{6}{\pi^2}\nonumber,\ee
where $\sigma_1=\prod_{p}\left(1-\frac{2}{p^2}\right)\approx0.3226340989$. 
We can also rewrite $\sigma_d=\sigma_1\cdot\prod_{p|d}\frac{1}{p^2-2}$. 

\begin{prop}\label{new-thm2&3}
Let $k$ be an arbitrary integer. Then 
\be c_2(k)=\sum_{\scriptsize{\ba{c}\mu^2(d)=1\\ d^2|k\ea}}\sigma_d.\label{statement-thm3}\ee
\end{prop}
\begin{proof}
Since 
$\mathcal P_2(k)=\{p:\: p^2|k\}$ and $\mathcal D(k)=\{\prod_{p\in\mathcal P'}p:\: \mathcal P'\subset\mathcal P_2(k)\}$, by Lemma \ref{prop-formula-sigma-d} gives 
\bey
\sum_{d\in \mathcal D(k)}\sigma_d&=&\sum_{d\in\mathcal D(k)}\frac{1}{d^2}\prod_{p\nmid d}\left(1-\frac{2}{p^2}\right)=\prod_{p}\left(1-\frac{2}{p^2}\right)\sum_{d\in\mathcal D(k)}\frac{1}{d^2}\prod_{p|d}\left(1-\frac{2}{p^2}\right)=\nonumber\\
&=&\prod_p\left(1-\frac{2}{p^2}\right)\sum_{d\in\mathcal D(k)}\prod_{p|d}\frac{1}{p^2-2}=\prod_{p}\left(1-\frac{2}{p^2}\right)\prod_{p\in\mathcal P_2(d)}\left(1+\frac{1}{p^2-2}\right)=\nonumber\\
&=&\prod_{p^2|k}\left(1-\frac{1}{p^2}\right)\prod_{p^2\nmid k}\left(1-\frac{2}{p^2}\right)=\,c_2(k)\nonumber
\eey
by (\ref{Mirsky-c2}).
\end{proof}

In particular, if $k=0$, then $\mathcal D(0)=\mathcal Q$ and $\mathcal P_2(0)=\mathcal P$ and we retrieve the known fact 
\bey
c_2(0)&=&\sum_{\mu^2(d)=1}\sigma_d=
\prod_p\left(1-\frac{1}{p^2}\right)=\frac{6}{\pi^2}.\nonumber
\eey

\begin{remark}
Proposition \ref{new-thm2&3} 
shows that the value of $c_2(k)$ depends on the arithmetic properties of $k$. This fact is certainly very unusual from the point of view of Probability Theory and Statistical Mechanics. 
If $k$ is square-free, then the function $c_2(k)$ takes the constant value $\sigma_1$. 
Analogously, $c_2(k)$ is constant along any subsequence of numbers $k$ sharing the same set of divisors that are the square of a square-free number. If we define $\mathcal D(k):=\{d:\:\mu^2(d)=1,\:d^2|k\}$, then $\mathcal D(k)=\mathcal D(k')\Rightarrow c_2(k)=c_2(k')$. The opposite implication follows from the formula (\ref{formula-sigma_d}). 
Observe that every set $\mathcal D(k)$ is of the form 
\be\mathcal D(k)=\left\{\prod_{p\in\mathcal P'}p:\:\mathcal P'\subseteq\mathcal P_2(k)\right\},\ee where 
$ \mathcal P_2(k)=\{p:\: p^2| k\}$. This means that $|\mathcal D(k)|=2^{|\mathcal P_2(k)|}$ and $\mathcal D(k)=\mathcal D(k')\iff\mathcal P_2(k)=\mathcal P_2(k')$.
The set of $k$ such that $\mathcal P_2(k)=\varnothing$ is the set of square-free numbers, and we know that it has positive density (equal to $6/\pi^2$, given by (\ref{density-square-free})). 
In general, we have the following
\end{remark}
\begin{prop}[Density of the level sets of $c_2$]\label{thm-density-level-sets-for-c2}
Fix a square-free number $d$.  
Then the density of those $k$'s such that $c_2(k)=c_2(d^2)$ exists and is given by
\be \frak d(d^2):=\lim_{N\to\infty}\frac{1}{N}\left|\left\{k\leq N:\: c_2(k)=c_2(d^2)\right\}\right|=\frac{6}{\pi^2}\prod_{p|d}\frac{1}{p^2-1}.\label{statement-thm-density-level-sets-for-c2}\ee
\end{prop}
\begin{proof}
If $\mathcal P(d)=\{p_1,p_2,\ldots,p_m\}$ then $k$ satisfies $c_2(k)=c_2(d^2)$ if and only if it is of the form $k=p_1^{a_1}\cdots p_m^{a_m}q$, where  $\mu^2(q)=1$, $a_j\geq2$, and $p_j\nmid q$ for every $j=1,\ldots,m$.
Fix $a_1,\ldots,a_m\geq 2$. Then 
\be\frac{1}{N}\left|\left\{k\leq N,\:k=p_1^{a_1}\cdots p_m^{a_m}q,\: \mu^2(q)=1,\: p_j\nmid q\:\mbox{for $j=1,\ldots,m$} \right\}\right|= \frac{1}{p_1^{a_1}\cdots p_m^{a_m}}\mathbb E_{(N/(p_1^{a_1}\cdots p_m^{a_m}))}^{\{p_1,\ldots,p_m\}}(0)\nonumber\ee
and, by Theorem \ref{thm1''}, the limit as $N\to\infty$ is
\be\frac{6}{\pi^2}\frac{1}{p_1^{a_1}\cdots p_m^{a_m}}\prod_{j=1}^m\frac{p_j}{p_j+1}=\prod_{j=1}^m\frac{1}{p_j^{a_j-1}(p_j+1)}\nonumber\ee
Now, by summing over all $a_j\geq 2$, we obtain
\be\frac{6}{\pi^2}\prod_{j=1}^m\frac{1}{p_j+1}\sum_{a_j\geq 2}\frac{1}{p_j^{a_j-1}}=\frac{6}{\pi^2}\prod_{j=1}^m\frac{1}{(p_j+1)(p_j-1)}=\frac{6}{\pi^2}\prod_{j=1}^m\frac{1}{p_j^2-1}\nonumber\ee
and the proposition is proven.
\end{proof}
\begin{remark}
The argument in the proof of Proposition  \ref{thm-density-level-sets-for-c2} will also be used in Appendix \ref{appendix-B}. We can check that
\be\sum_{\mu^2(d)=1}\frak d(d^2)=\sum_{m\geq0}\sum_{\scriptsize{\ba{c} \mathcal P'\subset \mathcal P\\|\mathcal P'|=m\ea}}\frac{6}{\pi^2}\prod_{p\in\mathcal P'}\frac{1}{p^2-1}=\frac{6}{\pi^2}\prod_p\left(1+\frac{1}{p^2-1}\right)=\frac{6}{\pi^2}\prod_{p}\left(1-\frac{1}{p^2}\right)^{-1}=1.\nonumber\ee
\end{remark}
Here we present the values of $\frak d(d^2)$ for square-free numbers $d\leq 17$. The sum of the corresponding densities is $\approx 97.6\%$ and one can check that $\sum_{d\leq 42:\: \mu^2(d)=1}\frak d(d^2)>99\%$.\\\\ 
\begin{tabular}{|c||c|c|c|c|c|c|c|c|c|c|c|c|c|}
\hline
$d$&1&2&3&5&6&7&10&11&13&14&15&17&$\ldots$\\
\hline
&&&&&&&&&&&&&\\
$\frak d(d^2)$&$\displaystyle{\frac{6}{\pi^2}}$&$\displaystyle{\frac{2}{\pi^2}}$&$\displaystyle{\frac{3}{4\pi^2}}$&$\displaystyle{\frac{1}{4\pi^2}}$&$\displaystyle{\frac{1}{4\pi^2}}$&$\displaystyle{\frac{1}{8\pi^2}}$&$\displaystyle{\frac{1}{12\pi^2}}$&$\displaystyle{\frac{1}{20\pi^2}}$&$\displaystyle{\frac{1}{28\pi^2}}$&$\displaystyle{\frac{1}{24\pi^2}}$&$\displaystyle{\frac{1}{32\pi^2}}$&$\displaystyle{\frac{1}{48\pi^2}}$&$\ldots$\\
&&&&&&&&&&&&&\\
\hline
\end{tabular}\\\\

One can also compute the limit 
\be\lim_{N\to\infty}\frac{1}{N}\sum_{n=1}^N c_2(n)=\left(\frac{6}{\pi^2}\right)^2\approx 0.3695753612\label{average-c2}
\ee
by considering the series $\sum_{\mu^2(d)=1}\frak d(d^2)c_2(d^2)$ and using Proposition \ref{thm-density-level-sets-for-c2} and Lemma \ref{prop-formula-sigma-d}. We shall retrieve this fact from the more general result of Lemma \ref{lem-expo-sum-c_2}. Figure \ref{figurec2} summarizes the structure of the second correlation function.

\begin{figure}[h!]
\begin{center}
\includegraphics[width=17cm]{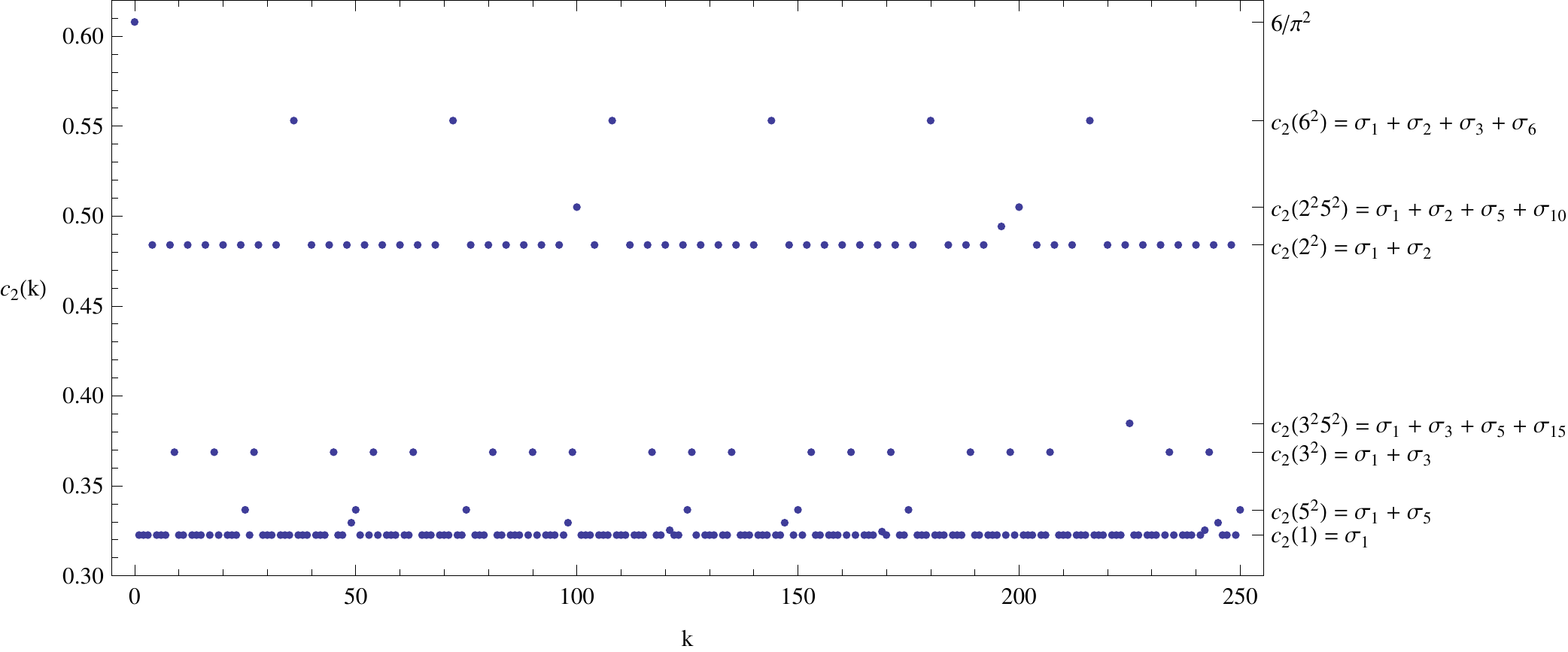}
\end{center}
\vspace{-.9cm}\caption{The second correlation function $c_2(k)$ and its level sets.}
\label{figurec2}
\end{figure}

Let us address the case of higher order correlation functions. 
\begin{prop}\label{prop-new-thm4&5}
Let $k_1,\ldots,k_r$ be such that all the differences 
$k_{l'}-k_{l''}$, $0\leq l'<l''\leq r
$ are square-free. Then 
\be c_{r+1}(k_1,k_2,\ldots,k_r)=\sum_{m_0,m_1,\ldots,m_r\geq0}(-1)^{\sum_{l=0}^r m_l}\sum_{\scriptsize{\ba{c}\mathcal P_l\subset \mathcal P,\:0\leq l\leq r\\ \mathcal P_{l'}\cap\mathcal P_{l''}=\varnothing\:\mbox{for}\:l'\neq l''\ea}}\frac{1}{\left[\bigcup_{l=0}^r \mathcal P_l\right]^2}\label{c_r+1-sinai}.\ee

For general $k_1,k_2,\ldots k_r$ we have 
\bey c_{r+1}(k_1,\ldots,k_r)=\sum_{0\leq l'<l''\leq r}\sum_{\scriptsize{\ba{c}\mu^2(d_{l',l''})=1\\ d_{l',l''}^2|k_{l'}-k_{l''}\ea}}\sum_{m_0,m_1,\ldots,m_r\geq 0}(-1)^{\sum_{l=0}^r m_l}\sum_{\scriptsize{\ba{c}\mathcal P_l,\: 0\leq l\leq r \\ |\mathcal P_l|=m_l\\\mbox{$[\mathcal P_{l'}\cap\mathcal P_{l''}]=d_{l',l''}$}\ea}}\frac{1}{\left[\bigcup_{l=0}^r \mathcal P_l\right]^2}.\label{general-c_r+1-sinai}\eey
\end{prop}
\begin{proof}
The case of $A_p^{(r+1)}(k_1,\ldots,k_r)=r+1$ corresponds to the case when $0, k_1, \ldots, k_r$ are distinct modulo $p^2$ for every prime $p$. This means that the differences $k_{l'}-k_{l''}$ are not divisible by $p^2$ for every prime $p$. In other words, the differences $k_{l'}-k_{l''}$ are all square-free. 
In this case, by writing $\mathcal P'=\cup_{l=0}^r\mathcal P_l$ and $m=m_0+\ldots+m_r$, the rhs of (\ref{c_r+1-sinai}) equals


\bey 
&&\sum_{m\geq 0}(-1)^m\frac{m!}{m_0!m_1!\cdots m_r!}\sum_{\scriptsize{\ba{c}\mathcal P'\subset\mathcal P\\ |\mathcal P|=m\ea}}\frac{1}{[\mathcal P']^2}
=\prod_{p}\left(1-\frac{r+1}{p^2}\right).\nonumber
\eey
Analogously, one can check that the formula in the rhs (\ref{general-c_r+1-sinai}) equals 
(\ref{general-formula-c_r+1-Mirsky}) with no restrictions on $k_1,\ldots,k_r$.
\end{proof}
\begin{remark}
Notice that the rhs of (\ref{c_r+1-sinai}) depends neither on $k_1,\ldots,k_r$ nor on the values of their differences as long as they all are square-free. Moreover, 
it is not enough to check that the  consecutive differences  $k_1-k_0,k_2-k_1,\ldots,k_r-k_{r-1}$ are square-free in order for all differences to be square-free.
For example, if $(k_1,k_2,k_3,k_4)=(1,6,7,10)$, all consecutive differences are square-free but $2^2|k_4-k_2$ and $3^2|k_4-k_1$. 

Notice also  that $c_{r+1}(k_1,\ldots, k_r)$ might be zero if $r\geq3$. For example, $c_4(1,2,3)=0$ since there is no $n$ such that $n,n+1,n+2,n+3$ are all square-free. All cases when $c_{r+1}(k_1,\ldots, k_r)=0$ correspond to constraints modulo $p^2$ for some prime $p$. This fact is clearly reflected by the general formula (\ref{general-formula-c_r+1-Mirsky}) for $c_{r+1}(k_1,\ldots,k_r)$. 
\end{remark}

Let us point out that the formul\ae\, (\ref{statement-thm3}, \ref{c_r+1-sinai}, \ref{general-c_r+1-sinai}) could be derived directly by inclusion-exclusion using arithmetic progressions with step $p^2$. That approach --as pointed out by the anonymous referee-- generates an error term that cannot be estimated elementarily. We therefore prefer the derivation of the formul\ae\, directly from Mirsky's.


%
In Section \ref{sec-averages-corr-fun} we shall use the following Lemma by 
R.R. Hall \cite{Hall-1989}. 
\begin{lem}
For every $0\leq k_1<k_2<\ldots<k_r$ we have 
\label{lem-Hall}
\be
c_{r+1}(k_1,\ldots,k_r)=\sum_{s_0\geq 1}\cdots\sum_{s_r\geq 1} g(s_0)\cdots g(s_r)\sum_{\scriptsize{\ba{c}0\leq t_j\leq s_j^2-1\\
\mu^2(\gcd(t_j,s_j^2))=1\\ 0\leq j\leq r\\ \frac{t_0}{s_0^2}+\frac{t_1}{s_1^2}+\ldots+\frac{t_r}{s_r^2}\in\Z\ea}}e\!\left(\frac{t_1}{s_1^2}k_1+\ldots+\frac{t_r}{s_r^2}k_r\right),\label{formula-Hall}
\ee
where
\be g(s)=\frac{6}{\pi^2}\mu(s)\prod_{p|s}\frac{1}{p^2-1}.\label{function-g-Hall}\ee
Moreover, the series in (\ref{formula-Hall}) converges absolutely.
\end{lem}

\subsection{Spectral analysis of $c_2$}
Let us expand slightly the discussion given in Section \ref{subsec-corr-fun}. We can rewrite (\ref{statement-thm3}) \be c_2(k)=\sum_{\mu^2(d)=1}K_d(k),\hspace{.5cm}\mbox{where}\hspace{.3cm} 
K_d(k):=\begin{cases}\sigma_d;&\mbox{if $d^2|k$};\\0,&\mbox{otherwise}.\end{cases}
\ee
The function $K_d$ is constant (equal to $\sigma_d$) along the arithmetic progression $\{ld^2:\:l\in\Z\}$ 
and 0 elsewhere. This function is the Fourier transform of a measure on the circle $\mathbb S^1$, given by a sum of $\delta$-functions at the points $e\!\left(l/d^2\right)$, $l=0,1,\ldots,d^2-1$, with equal weights $\sigma_d/d^2$. A corollary of this fact is the formula (\ref{spectral-measure-nu}) for the spectral measure $\nu$ on $\mathbb S^1$. 

\section{Averages of the correlation functions}\label{sec-averages-corr-fun}
This section is dedicated to the proof of some results generalizing (\ref{average-c2}). For instance, one can restrict the average to those integers belonging to a certain residue class modulo a square-free $d$ (Lemmata \ref{lem-sum-c_2(d^2l)}-
\ref{lem-sum-c_2(d^2l+t)-general}). These averages are then used in the analysis of an exponential sums of the form $\frac{1}{N}\sum_{n=1}^N \lambda^n c_2(n)$, 
where $\lambda$ is a complex number of modulo 1 (Lemma \ref{lem-expo-sum-c_2}) in the case when $\lambda\in\Lambda$. The latter case be further extended to multiple averages of higher-order correlation functions (Lemma \ref{lem-sums-new}). 
These exponential sums play a crucial role in the spectral analysis of the Koopman operator for the `natural' dynamical systems $(X,\mathcal B,\Pi,T)$ from Section \ref{natural-dyn-sys}, whose invariant measure is defined by means of the correlations $c_{r+1}(k_1,\ldots,k_r)$ (see Section \ref{sec-spectrum-of-T}). For example, given an eigenfunction $\theta_\lambda: X\to\C$  with eigenvalue $\lambda\in\Lambda$ for the Koopman operator, we will see that its correlation with the projection onto the $s$-th coordinate $x\to x(s)\in\{0,1\}$, i.e. the inner product $\langle x(s), \theta_\lambda\rangle_{L^2(X,\mathcal B, \Pi)}$, is given by $\lambda^s\lim_{N\to\infty}\frac{1}{N}\sum_{n=1}^N\lambda^n c_2(n)$, and we will use the the explicit form of this limit as function of $\lambda$ to study the set of all eigenfunctions $\{\theta_\lambda\}_{\lambda\in\Lambda}$.
%
%
%
The proofs of the first three lemmata are given in Appendix \ref{appendix-B}.
\begin{lem}\label{lem-sum-c_2(d^2l)}
Let $d$ be square-free. Then
\be\lim_{n\to\infty}\frac{1}{N}\sum_{l\leq N}c_2(d^2 l)=\left(\frac{6}{\pi^2}\right)^2\prod_{p\in\mathcal P(d)}\frac{p^2}{p^2-1}.\nonumber\ee
\end{lem}
\begin{lem}\label{lem-sum_c_2(d^l+t)-square-free}
Let $d$ be square-free and let $1\leq t\leq d^2-1$, $\gcd(d^2,t)=g\geq 1$, where $g$ is square-free. Then
\bey
\lim_{n\to\infty}\frac{1}{N}\sum_{l\leq N}c_2(d^2 l+t)=\left(\frac{6}{\pi^2}\right)^2\prod_{p\in\mathcal P(d)}\frac{p^2(p^2-2)}{(p^2-1)^2}.\nonumber
\eey
\end{lem}
\begin{lem}\label{lem-sum-c_2(d^2l+t)-general}
Let $d$ be square-free, and let $1\leq t\leq d^2-1$, $\gcd(t,d^2)=g\geq 1$. Then
\bey
\lim_{n\to\infty}\frac{1}{N}\sum_{l\leq N}c_2(d^2 l+t)=\left(\frac{6}{\pi^2}\right)^2\prod_{p\in\mathcal P(d)}\frac{p^2(p^2-2)}{(p^2-1)^2}\prod_{p\in\mathcal P_2(g)}\frac{p^2-1}{p^2-2}.\nonumber
\eey
\end{lem}
The following two lemmata deal with exponential sums involving the second and the third correlation functions. Recall the function $g$ from (\ref{function-g-Hall}).
\begin{lem}\label{lem-expo-sum-c_2}
Let $\lambda=e(\tfrac{l}{d^2})\in\Lambda$.
Then the limit
\be
\frak Y_2(\lambda)=\lim_{N\to\infty}\frac{1}{N}\sum_{n=1}^N \lambda^n 
c_2(n)
\nonumber\ee
exists and $\frak Y_2(\lambda)=g(d)^2$.
\end{lem}
\begin{proof}
We can write $n=d^2l+t$ for some $l\geq 0$ and $0\leq t\leq d^2-1$ and set
\be I_N(\lambda 
)=\frac{1}{N}\sum_{n\leq N}\lambda^n 
c_2(n)=\sum_{t=0}^{d^2-1}I_N^{(t)}(l,d^2),\label{pf-innerproduc-b}\ee
where 
\be I_N^{(t)}(\lambda 
)=\frac{1}{N}\sum_{m\leq (N-t)/d^2}e\left(\frac{lt}{d^2}\right)
c_2(d^2 m+t).\nonumber\ee

Lemma \ref{lem-sum-c_2(d^2l)} gives
\be\lim_{N\to\infty}I_N^{(0)}(\lambda
)=\frac{1}{d^2} \left(\frac{6}{\pi^2}\right)^2\prod_{p|d}\frac{p^2}{p^2-1}.\label{asymptotic-I_N^(0)(d)}\ee
For $t\neq 0$, the value of $\lim_{N\to\infty} I_N^{(t)}(\lambda
)$ is given by Lemmata \ref{lem-sum_c_2(d^l+t)-square-free}-\ref{lem-sum-c_2(d^2l+t)-general}. 
It depends only on $\mathcal P_2(g)$, where $g=\gcd(t,d^2)$. More explicitly,
\be \lim_{N\to\infty}I_N^{(t)}(\lambda 
)=e\!\left(\frac{l t}{d^2}\right)
\frac{1}{d^2}\left(\frac{6}{\pi^2}\right)^2\prod_{p|d}\frac{p^2(p^2-2)}{(p^2-1)^2}\prod_{p^2|g}\frac{p^2-1}{p^2-2}\label{asymptotic-I_N^(t)(d)}.\ee
Let us introduce the notation
\be\tau_t(d^2)=\prod_{p^2|\gcd(t,d^2)}\frac{p^2-1}{p^2-2}.\nonumber\ee
Notice that 
$\tau_{t}(d^2)=\tau_{d^2-t}(d^2)$ and therefore the limit $\lim_{N\to\infty}I_N(\lambda
)$ is real.
Using (\ref{asymptotic-I_N^(0)(d)}) and (\ref{asymptotic-I_N^(t)(d)}) we can write 
\be\lim_{N\to\infty}I_N(\lambda 
)=\left(\frac{6}{\pi^2}\right)^2\prod_{p|d}\frac{1}{p^2-1}\left(1+\prod_{p|d}\frac{p^2-2}{p^2-1}\sum_{t=1}^{d^2-1}\cos\left(\frac{2\pi l t}{d^2}\right)\tau_t(d^2)\right)\label{value-lim-I_N(d)}\ee
and, if $\omega(d)=|\mathcal P(d)|=r$, then 
\bey
&&\sum_{t=1}^{d^2-1}\cos\!\left(\frac{2\pi l t}{d^2}\right)\tau_t(d^2)=
\sum_{\scriptsize{\ba{c}t\leq d^2-1
\ea}}\cos\!\left(\frac{2\pi l t}{d^2}\right)-\sum_{p_1|d}\sum_{\scriptsize{\ba{c}t\leq d^2-1\\p_1^2|t\ea}}\cos\!\left(\frac{2\pi l t}{d^2}\right)\left(1-\frac{p_1^2-1}{p_1^2-2}\right)+\nonumber\\
&&+\sum_{p_1,p_2|d}\sum_{\scriptsize{\ba{c}t\leq d^2-1\\p_1^2p_2^2|t\ea}}\cos\!\left(\frac{2\pi l t}{d^2}\right)\left(1-\frac{p_1^2-1}{p_1^2-2}\right)\left(1-\frac{p_2^2-1}{p_2^2-2}\right)-\nonumber\\
&&-\sum_{p_1,p_2,p_3|d}\sum_{\scriptsize{\ba{c}t\leq d^2-1\\p_1^2p_2^2p^3|t\ea}}\cos\!\left(\frac{2\pi l t}{d^2}\right)\left(1-\frac{p_1^2-1}{p_1^2-2}\right)\left(1-\frac{p_2^2-1}{p_2^2-2}\right)\left(1-\frac{p_3^2-1}{p_3^2-2}\right)+\ldots+\nonumber\\
&&+(-1)^{r-1}\sum_{p_1,\ldots,p_{r-1}|d}\sum_{\scriptsize{\ba{c}t\leq d^2-1\\p_1^2\cdots p_{r-1}^2|t\ea}}\cos\!\left(\frac{2\pi l t}{d^2}\right)\left(1-\frac{p_1^2-1}{p_1^2-2}\right)\cdots\left(1-\frac{p_{r-1}^2-1}{p_{r-1}^2-2}\right)\label{computation-sum-cosines}.
\eey
Recall the assumption that $\gcd(l,d^2)$ is square-free, and notice that for every $p_1,\ldots,p_m|d$, $m<r$, 
\be
\sum_{\scriptsize{\ba{c}t\leq d^2-1\\p_1^2\cdots p_{m}^2|t\ea}}\cos\!\left(\frac{2\pi l t}{d^2}\right)=-1\nonumber. 
\ee
Now (\ref{computation-sum-cosines}) yields 
\bey
&&\sum_{t=1}^{d^2-1}\cos\!\left(\frac{2\pi l t}{d^2}\right)\tau_t(d^2)=(-1)\left(1-\sum_{p_1|d}\frac{-1}{p_1^2-2}+\sum_{p_1,p_2|d}\frac{-1}{p_1^2-2}\frac{-1}{p_2^2-2}-\right.\nonumber\\
&&\left.-\sum_{p_1,p_2,p_3|d}\frac{-1}{p_1^2-2}\frac{-1}{p_2^2-2}\frac{-1}{p_3^2-2}+\ldots+(-1)^{r-1}\sum_{p_1,\ldots,p_{r-1}|d}\frac{-1}{p_1^2-2}\cdots\frac{-1}{p_{r-1}^2-2}\right)=\nonumber\\
&&=
-\left(\prod_{p|d}\left(1-\frac{-1}{p^2-2}\right)-(-1)^r\prod_{p|d}\frac{-1}{p^2-2} \right)=\prod_{p|d}\frac{1}{p^2-2}-\prod_{p|d}\frac{p^2-1}{p^2-2},
\nonumber
\eey
and (\ref{value-lim-I_N(d)}) becomes
\be\lim_{N\to\infty}I_N(\lambda 
)=\left(\frac{6}{\pi^2}\right)^2\prod_{p|d}\frac{1}{p^2-1}\left(1+\prod_{p|d}\frac{1}{p^2-1}-1\right)=\left(\frac{6}{\pi^2}\right)^2\prod_{p|d}\frac{1}{(p^2-1)^2}=g(d)^2.\nonumber 
\ee
\end{proof}

\begin{remark}\label{rk-bigO}
Since $\frac{3}{4} p^2\leq p^2-1\leq p^2$ for every $p$, then 
\be\frac{1}{d^2}\leq \prod_{p|d}\frac{1}{p^2-1}\leq \left(\frac{4}{3}\right)^{\omega(d)}\frac{1}{d^2}.\ee
Since $d$ is square-free, if we want to give an upper bound for $\omega(d)$ in terms of $d$ as $d\to\infty$, it is enough to consider the case when $d$ is the product of the first $r$ prime numbers: $d=p_1p_2\cdots p_r$. In this case $\omega(d)=r$. It is known that $\log d=r\log r(1+o(1))$ as $r\to\infty$. This means that in general $\omega(d)\log \omega(d)\leq (1+\varepsilon_1)\log d$ for every $\varepsilon_1>0$, provided that $d\gg 1$. This implies $\omega(d)\leq\frac{\log d^{1+\varepsilon}}{W(\log d^{1+\varepsilon_1})}$, where $W$ denotes the Lambert function, i.e. the solution of the equation $x=W(x)e^{W(x)}$. It is known that $W(x)\sim \log x$ as $x\to\infty$. Therefore $\omega(d)\leq \frac{\log d^{1+\varepsilon_1}}{\log(\log d^{1+\varepsilon_1})^{1-\varepsilon_2}}$ for every $\varepsilon_2>0$ and thus
\be \left(\frac{4}{3}\right)^{\omega(d)}=\mathcal{O}_\varepsilon(d^\varepsilon)\ee for every $\varepsilon>0$ as $d\to\infty$. In other words, formul\ae (\ref{asymptotic-I_N^(0)(d)}-\ref{asymptotic-I_N^(t)(d)}) give
\be\lim_{N\to\infty}I_N^{(t)}(l,d^2)=\mathcal{O}_\varepsilon\!\left(\frac{1}{d^{2-\varepsilon}}\right)\hspace{.5cm}\mbox{as $d\to\infty$},\nonumber\ee
for every $t=0,1,\ldots, d^2-1$.
However, 
the cancellations coming from the different exponential factors $e^{\frac{2\pi i t}{d^2}}$ in $I_N(d)$ are responsible for the higher order of smallness shown in (\ref{value-lim-I_N(d)}): \be\lim_{N\to\infty}I_N(l,d^2)=\mathcal O\!\left(\frac{1}{d^{4-\varepsilon}}\right)\hspace{.5cm}\mbox{as $d\to\infty$}.\nonumber\ee
\end{remark}

\begin{lem}\label{lem-sums-new}
Let $\lambda_1,\lambda_2\in\Lambda$, $\lambda_1=e\!\left(\frac{l_1}{d_1^2}\right), \lambda_2=e\!\left(\frac{l_2}{d_2^2}\right)$, and $\lambda=\lambda_1\lambda_2
=e\!\left(\frac{l}{d^2}\right)
\in\Lambda$.
Then the 2-fold limit 
\bey
\frak Y_3(\lambda_1,\lambda_2)=
\lim_{\scriptsize{\ba{c}N_1\to\infty\\ N_2\to\infty
\ea}}\frac{1}{N_1N_2}\sum_{n_1=1}^{N_1}\sum_{n_2=1}^{N_2}\lambda_1^{n_2}\lambda_2^{n_2}
c_{3}(n_1,n_2)
\nonumber\\
\eey
exists and $\frak Y_3(\lambda_1,\lambda_2)=g(d_1)g(d_2)g(d)$.
\end{lem}
\begin{proof}
Using Lemma \ref{lem-Hall} we can write
\bey
\frak Y_{3}(\lambda_1,\lambda_2)&=&
\lim_{\scriptsize{\ba{c}N_1\to\infty\\N_2\to\infty\ea}}\frac{1}{N_1N_2}\sum_{n_1=1}^{N_1}\sum_{n_2=1}^{N_2}\lambda_1^{n_1}\lambda_2^{n_2} 
\sum_{s_0\geq1}\sum_{s_1\geq 1}\sum_{s_2\geq 1}g(s_0)g(s_1)g(s_2)\nonumber\\
&&
\sum_{\scriptsize{\ba{c}0\leq t_j\leq s_j^2-1\\
\mu^2(\gcd(t_j,s_j^2))=1\\
j=0,1,2\\
\frac{t_0}{s_0^2}+\frac{t_1}{s_1^2}+\frac{t_2}{s_2^2}\in\Z\ea}}e\!\left(\frac{t_1}{s_1^2}n_1+\frac{t_2}{s_2^2}n_2\right).\label{pf-lemma-prod-p1-p2-1}
\eey
Let us bring the limit and the sums over $n_1, n_2$ in (\ref{pf-lemma-prod-p1-p2-1}) inside the sum over $t_0,t_1,t_2$.
For fixed $s_0,s_1,s_2$ we have
\bey
\sum_{\scriptsize{\ba{c}0\leq t_j\leq s_j^2-1\\
\mu^2(\gcd(t_j,s_j^2))=1\\
j=0,1,2\\
\frac{t_0}{s_0^2}+\frac{t_1}{s_1^2}+\frac{t_2}{s_2^2}\in\Z\ea}}
\left(\lim_{N_1\to\infty}\frac{1}{N_1}\sum_{n_1=1}^{N_1} e\!\left(\left(\frac{t_1}{s_1^2}+\frac{l_1}{d_1^2}\right)n_1\right) \right)
\left(\lim_{N_2\to\infty}\sum_{n_2=1}^{N_2}
e\!\left(\left(\frac{t_2}{s_2^2}n_2+\frac{l_2}{d_2^2}\right)n_2\right)\right)\nonumber 
\eey
The two sums over $n_1,n_2$ can be written as
\be
\frac{1}{N_j}\sum_{n_j=1}^{N_j}e\!\left(\left(\frac{t_j}{s_j^2}+\frac{l_j}{d_j^2}\right)n_j\right)=\begin{cases}
\displaystyle{\frac{1}{N_j}\frac{\xi_j-\xi_j^{N+1}}{1-\xi_j}},&\mbox{if $\xi_j\neq 1$}\\ &\\
\displaystyle{1},&\mbox{if $\xi_j=1$,}
\end{cases}\nonumber
\ee
where $\xi_j= e\!\left(\frac{t_j}{s_j^2}+\frac{l_j}{d_j^2}\right)$ and $j=1,2$. Thus, as $N\to\infty$, only the indices $t_0,t_1,t_2$ such that $\xi_1=\xi_2=1$ give a non-zero contribution to 
(\ref{pf-lemma-prod-p1-p2-1}). This condition means
\bey
\frac{t_j}{s_j^2}+\frac{l_j}{d_j^2}=\begin{cases}1,&\mbox{if $\lambda_j\neq 1$;}\\ 0,&\mbox{if $\lambda_j=1$;}\end{cases}\nonumber
\eey
for $j=1,2$. However, 
because of the conditions $0\leq t_j\leq s_j^2-1$ and $\mu^2(\gcd(t_j,s_j^2))=1$, the index 
\be t_j=\begin{cases}\frac{s_j^2}{d_j^2}\left(d_j^2-l_j\right),&\mbox{if $\lambda_j\neq1$}\\ 0,&\mbox{if $\lambda_j=1$}\end{cases}\nonumber\ee 
will be considered only when $s_j=d_j$. The value of  $t_0$ is given consequently by
\be \frac{t_0}{s_0^2}=\begin{cases}\frac{l}{d^2}, &\mbox{if $\lambda_1\neq 1\neq \lambda_2$;}\\ \frac{l_1}{d_1^2}.&\mbox{if $\lambda_1\neq1=\lambda_2$}\\ \frac{l_2}{d_2^2},&\mbox{if $\lambda_1=1\neq\lambda_2$}\\0, &\mbox{if $\lambda_1=1=\lambda_2$}. \end{cases}
\nonumber\ee
In all cases this means $\frac{t_0}{s_0^2}=\frac{l}{d^2}$, and 
%
the condition $\mu^2(\gcd(t_0,s_0^2))=1$ implies that $s_0=d$ and $t_0=l$. In other words, (\ref{pf-lemma-prod-p1-p2-1}) becomes $\frak Y_3(\lambda_1,\lambda_2)=g(d_1)g(d_2)g(d)$, and the lemma is proven.
\end{proof}
\begin{remark}
Notice that if $\lambda_1=e\!\left(\frac{l_1}{d_1^2}\right),\lambda_2=e\!\left(\frac{l_2}{d_2^2}\right)\in\Lambda$ satisfy $\gcd(d_1,d_2)=1$, then  
\be\frak Y_3\!\left(\lambda_1,\lambda_2 
\right)=
\left(\frac{6}{\pi^2}\right)^3(-1)^{\omega(d_1d_2)}\prod_{p|d_1d_2}\frac{1}{(p^2-1)^2}.
\nonumber\ee
\end{remark}
The product $\prod_{p|d}(p^2-1)$ appears in several formul\ae\, above. Concerning this product, we have the following basic
\begin{lem}\label{lem-prod-p-set-of-l} Let $d$ be square-free. Then
\be\prod_{p|d}(p^2-1)=\left|\left\{1\leq l\leq d^2-1,\:\mu^2(\gcd(l,d^2))=1\right\}\right|\label{prod-p-set-of-l}\ee
\end{lem}
\begin{proof}
By standard inclusion-exclusion we can write the rhs of (\ref{prod-p-set-of-l}) 
\bey
 &&d^2-\sum_{p_1|d}\frac{d^2}{p_1^2}+\sum_{p_1,p_2|d}\frac{d^2}{(p_1p_2)^2}-\sum_{p_1,p_2,p_3|d}\frac{d^2}{(p_1p_2p_3)^2}+\ldots+(-1)^{\omega(d)}\nonumber=\\ 
&&=d^2\prod_{p|d}\left(1-\frac{1}{p^2}\right)=\prod_{p|d}(p^2-1).\nonumber
\eey
\end{proof}

\section{The spectrum of the shift operator $T$ 
}\label{sec-spectrum-of-T}
Recall the definition of the dynamical system $(X,\mathcal B,\Pi,T)$ given in Section \ref{natural-dyn-sys}.
Denote by $U$ the operator on the Hilbert space $\mathcal H=L^2(X,\mathcal B,\Pi)$ given by
\be (U f)(x)=f(T x)\label{def-operator-U}.\ee
Since $T$ is measure-preserving, the operator $U$ is unitary. The goal of this section is to prove the following
\begin{theorem}\label{thm-spectrum-of-U-is-Lambda}
The spectrum of the operator $U$ is given by $\Lambda$.
\end{theorem}
Let us show that $\Lambda$ is contained in the spectrum of $U$. 
This fact is given by the following
\begin{theorem}\label{existence-of-eigenfunctions}
Let $\lambda=e\!\left(\frac{l}{d^2}\right)\in\Lambda$. 
Then there exists a function $\theta_{\lambda
}\in\mathcal H$ such that
\be (U\theta_{\lambda
})(x)= \lambda\, 
\theta_{\lambda 
}(x)\label{general-construction-eigenfunctions}\ee
for $\Pi$-almost every $x\in X$.
\end{theorem}
\begin{proof}
Let $f_0(x)=x(0)$ and let $U_{\lambda 
}$ be the unitary operator on $\mathcal H$ defined by
\be(U_{\lambda
} h)(x)=\lambda^{-1}\,
h(Tx).\nonumber\ee
By the von Neumann Ergodic Theorem, the following limit exists in $\mathcal H$:
\be\lim_{N\to\infty}\frac{1}{N}\sum_{n=1}^N U_{\lambda
}^nf_0(x)=\lim_{N\to\infty}\frac{1}{N}\sum_{n=1}^N \lambda^{-n} 
f_0(T^n x)=\lim_{N\to\infty}\frac{1}{N}\sum_{n=1}^N  \lambda^{-n} 
x(n)=:\theta_{\lambda 
}(x).\label{def-eigenfunction-theta_l,d^2}\ee
The function $\theta_{\lambda 
}$ is $U_{\lambda 
}$-invariant, i.e. $(U_{\lambda 
}\theta_{\lambda
})(x)=\theta_{\lambda
}(x)$ for $\Pi$-almost every $x\in X$. This implies that $\lambda^{-1}
\theta_{\lambda 
}(Tx)=\theta_{\lambda
}(x)$, i.e. (\ref{general-construction-eigenfunctions}).
\end{proof}

Denote by $x(s)$ the function $X\to\{0,1\}$ given by the projection of $x\in X$ onto its $s$-th coordinate. Introduce the subspace $H\subseteq \mathcal H$, 
\be H=\overline{\left\{\sum_{s} a_s x(s)\right\}},\nonumber\ee
i.e. 
the closure of the set of all complex linear combinations of the $x(s)$'s.
$H$ is invariant under $U$, and by (\ref{def-eigenfunction-theta_l,d^2}), all the eigenfunctions $\theta_{\lambda 
}
$ belong to $H$. Let us remark that, since the operator $U$ is unitary, the eigenfunctions $\theta_{\lambda}$ are orthogonal to one another for different $\lambda$.
%
%
Let us write \be x(s)=\sum_{\lambda\in\Lambda}\left\langle x(s),\theta_\lambda\right\rangle\theta_\lambda.\nonumber\ee 
Recall 
(\ref{function-g-Hall}).
We have the following 
\begin{theorem}\label{thm-innerproduct-theta_l,d^2-x(s)}
Let $\lambda=e\!\left(\frac{l}{d^2}\right)\in\Lambda$. Then for every $s\in\Z$ we have
\be \left\langle x(s), \theta_{\lambda
}
\right\rangle=
\lambda^s g(d)^2.
\label{statement-thm-innerproduct-theta_l,d^2-x(s)}\ee
\end{theorem}
%
%
\begin{proof}
Let us use (\ref{def-eigenfunction-theta_l,d^2}) and write
\bey
\left\langle x(s),\theta_{\lambda 
}
\right\rangle=\lim_{N\to\infty}\left\langle x(s),\frac{1}{N}\sum_{n=1}^N \lambda^{-n}
x(n)
\right\rangle=\lim_{N\to\infty}\frac{1}{N}\sum_{n=1}^N
\lambda^{n}
\langle x(s),x(n)\rangle.\label{pf-eigenfunctions-theta_d-1}\eey
Notice that $\langle x(s),x(n)\rangle=c_2(n-s)$, where $c_2$ is the second correlation function given by Proposition \ref{new-thm2&3}. Equation 
(\ref{pf-eigenfunctions-theta_d-1}) becomes
\bey
\langle x(s),\theta_{\lambda 
}
\rangle=\lim_{N\to\infty}\frac{1}{N}\sum_{n=1}^N \lambda^{n} 
c_2(n-s)
=\lambda^{s} 
\lim_{N\to\infty}\frac{1}{N}\sum_{n=1}^N \lambda^{n} 
c_2(n)=\lambda^s\,\frak Y_2(\lambda).\label{pf-innerproduct-a}
\eey
The needed statement follows now from Lemma \ref{lem-expo-sum-c_2}.
\end{proof}
Theorem \ref{thm-innerproduct-theta_l,d^2-x(s)} immediately implies the following
\begin{cor}\label{cor-eigenfunctions-nontrivial}
All eigenfunctions $\theta_{\lambda
}$ are non-zero.
\end{cor}

\begin{remark}
The formula (\ref{def-eigenfunction-theta_l,d^2}) can be written for arbitrary measure-preserving map, but in most cases (e.g. automorphisms with continuous spectrum) it gives zero. Theorem \ref{thm-innerproduct-theta_l,d^2-x(s)} shows that in our case it is non-zero.
\end{remark}

We can also compute the $L^2$-norm of each eigenfunction explicitly. \begin{theorem}\label{norm-eigenfunctions}
Let $\lambda=e\!\left(\frac{l}{d^2}\right)\in\Lambda$. 
Then
\be\left\|\theta_{\lambda
}\right\|=\left|g(d)\right|
.\ee
\end{theorem}
 \begin{proof}
This is a straightforward application of Theorem \ref{thm-innerproduct-theta_l,d^2-x(s)}:
\bey
\left\|\theta_{\lambda
}\right\|^2&=&\langle \theta_{\lambda
},\theta_{\lambda
}\rangle=\left\langle \theta_{\lambda
}, \lim_{N\to\infty}\frac{1}{N}\sum_{n=1}^{N}\lambda^{-n} 
x(n)\right\rangle=\lim_{N\to\infty}\frac{1}{N}\sum_{n=1}^{N}\lambda^n 
\overline{\langle x(s), \theta_{\lambda
}
\rangle}=\nonumber\\
&=&\left(\frac{6}{\pi^2}\right)^2\prod_{p|d}\frac{1}{(p^2-1)^2}.\nonumber\eey
 \end{proof}
%
%
%

\begin{prop}\label{eigenfuctions-basis-H}
The set of eigenfunctions $\{\theta_\lambda\}_{\lambda\in\Lambda}$ is a basis for $H$.
\end{prop}
\begin{proof}
Since the eigenfunctions are orthogonal it is enough to show that they span the space of all linear combinations of the $x(s)$'s. 
We know that each atom $\{\lambda\}$ of the spectral measure $\nu$ (associated to the second correlation function via  Bochner's theorem) corresponds to $\theta_\lambda$  in the space $H$ generated by linear forms, and these form a set of generators for $H$.
\end{proof}
Let us define the normalized eigenfunctions: for $\lambda\in\Lambda$ set
\be \tilde\theta_\lambda=\frac{\theta_\lambda}{\|\theta_\lambda\|},\nonumber\ee
so that $\{\tilde\theta_\lambda\}_{\lambda\in\Lambda}$ is an orthonormal basis for $H$.
Let us write
\bey x(s)=\sum_{\lambda\in\Lambda}\left\langle x(s),\tilde\theta_\lambda\right\rangle\tilde\theta_\lambda.\nonumber\eey
Since $\{\tilde\theta_\lambda\}_{\lambda}$ is an orthonormal basis for $H$, then by Lemma \ref{lem-prod-p-set-of-l} and Theorem \ref{norm-eigenfunctions}
\bey ||x(s)||^2&=&\sum_{\lambda\in\Lambda}\left|\left\langle x(s),\tilde\theta_\lambda\right\rangle\right|^2=\left(\frac{6}{\pi^2}\right)^2\sum_{d\in\mathcal Q}\prod_{p|d}\frac{1}{p^2-1}=\left(\frac{6}{\pi^2}\right)^2\sum_{\mathcal P'\subset\mathcal P,\, |\mathcal P'|<\infty}\prod_{p\in\mathcal P'}\frac{1}{p^2-1}=\nonumber\\
&=&\left(\frac{6}{\pi^2}\right)^2\prod_{p}\left(1+\frac{1}{p^2-1}\right)=\frac{6}{\pi^2}.\nonumber\eey
The same argument allows us to estimate the size of the error term in the following approximation of $x(s)$: for $D\geq 1$ let
\be x_{ D}(s)=\sum_{\scriptsize{\ba{c}\lambda=e\!\left(l/d^2\right)\in\Lambda\\ d\leq D\ea}}\left\langle x(s),\tilde\theta_{\lambda}\right\rangle\tilde\theta_\lambda\nonumber.\ee
Arguing as in Remark \ref{rk-bigO}, we have
\be
\left\| x(s)-x_D(s)\right\|^2=\sum_{\scriptsize{\ba{c} \lambda=e\!\left(l/d^2\right)\in\Lambda\\ d>D\ea}}\left|\left\langle x(s),\tilde\theta_\lambda\right\rangle\right|^2=\frac{6}{\pi^2}\sum_{d>D}|g(d)|=\mathcal O_{\varepsilon}(D^{-1+\varepsilon})\label{quality-approximation-xD(s)}
\ee
for every $\varepsilon>0$.

Let us consider the product of two 
 eigenfunctions $\tilde\theta_{\lambda_1
 }$ and $\tilde\theta_{\lambda_2
 }$.  We have the following
 \begin{theorem}\label{product-eigenfunctions} Let $\lambda_1=e\!\left(\frac{l_1}{d_1^2}\right), \lambda_2=e\!\left(\frac{l_2}{d_2^2}\right)\in\Lambda$. Then
 \be\tilde\theta_{\lambda_1
 }\tilde\theta_{\lambda_2
 }=
\epsilon\,\tilde\theta_{\lambda 
 },\ee
 where $\lambda=e\!\left(\frac{l}{d^2}\right)=\lambda_1\lambda_2
 $ and $\epsilon=\epsilon(\lambda_1,\lambda_2)=\mu(d_1)\mu(d_2)\mu(d)=\pm 1$. 
 \end{theorem}
 \begin{proof}
 It is enough to show that for every $s\in\Z$ we have
\be\left\langle\tilde\theta_{\lambda_1}\tilde\theta_{\lambda_2},x(s)\right\rangle=\left\langle\tilde\theta_\lambda,x(s)\right\rangle\nonumber\ee 
Using the definition (\ref{def-eigenfunction-theta_l,d^2}) we can write
\be 
\theta_{\lambda_1
}
\theta_{\lambda_2
}=\lim_{\scriptsize{\ba{c}N_1\to\infty\\ N_2\to\infty\ea}}\frac{1}{N_1N_2}\sum_{n=1}^{N_1}\sum_{n=2}^{N_2} 
\lambda_1^{-n_1}\lambda_2^{-n_2}
x(n_1)x(n_2)\nonumber\ee
and thus
\be
\left\langle \theta_{\lambda_1
}\theta_{\lambda_2
},x(s)\right\rangle=\lim_{\scriptsize{\ba{c}N_1\to\infty\\ N_2\to\infty\ea}}\frac{1}{N_1N_2}\sum_{n=1}^{N_1}\sum_{n=2}^{N_2} \lambda_1^{-n_1}\lambda_2^{-n_2}
\left\langle x(n_1)x(n_2),x(s)\right\rangle\nonumber.
\ee
Notice that $\left\langle x(n_1)x(n_2),x(s)\right\rangle=c_3(n_1-s,n_2-s)$. Therefore, by Lemma \ref{lem-sums-new},
\bey
\left\langle \theta_{\lambda_1
}\theta_{\lambda_2
},x(s)\right\rangle&=&\lim_{\scriptsize{\ba{c}N_1\to\infty\\ N_2\to\infty\ea}}\frac{1}{N_1N_2}\sum_{n=1}^{N_1}\sum_{n=2}^{N_2} 
 \lambda_1^{-n_1}\lambda_2^{-n_2}
c_3(n_1-s,n_2-s)=\nonumber\\
&=&\lambda^{-s} 
\lim_{\scriptsize{\ba{c}N_1\to\infty\\ N_2\to\infty\ea}}\frac{1}{N_1N_2}\sum_{n=1}^{N_1}\sum_{n=2}^{N_2} \lambda_1^{-n_1}\lambda_2^{-n_2}
c_3(n_1,n_2)=\nonumber\\
&=&\lambda^{-s}
\,\frak Y_{3}(\lambda_1^{-1},\lambda_2^{-1}
)=\lambda^{-s}\,g(d_1)g(d_2)g(d).\nonumber
\eey
On the other hand, by Theorem \ref{thm-innerproduct-theta_l,d^2-x(s)},
\be
\left\langle \theta_{\lambda
}, x(s)\right\rangle=e^{-\lambda}\,
\frak Y_2(\lambda
)=\lambda^{-s}\,g(d)^2.\nonumber 
\ee
Therefore
\be\epsilon= \left\langle \tilde\theta_{\lambda_1}\tilde\theta_{\lambda_2},x(s)\right\rangle\left\langle \tilde\theta_\lambda,x(s)\right\rangle^{-1}=\frac{g(d_1)g(d_2)g(d)}{|g(d_1)|\,|g(d_2)|}\frac{|g(d)|}{g(d)^2}=\mu(d_1)\mu(d_2)\mu(d).\nonumber\ee
 \end{proof}
By associativity of multiplication, $\epsilon(\lambda_1,\lambda_2)\epsilon(\lambda_1\lambda_2,\lambda_3)=\epsilon(\lambda_2,\lambda_3)\epsilon(\lambda_1,\lambda_2\lambda_3)$.
Theorem \ref{product-eigenfunctions} can be applied iteratively. It allows us to write all polynomial expressions in the  eigenfunctions as linear expressions, and this is a very important fact.

We want to show that the set of eigenfunctions $\{\theta_\lambda\}_{\lambda\in\Lambda}$ is a basis for the whole space $\mathcal H$. We shall need the notion of 
\emph{unitary rings} introduced by V.A. Rokhlin (see \cite{Rohlin-1948}).

%
\begin{defin}
A complex Hilbert space $\frak H$ is called a \emph{unitary ring} if and only if, for certain pairs of elements, a product is defined satisfying:
\begin{itemize}
\item[(I)] if $fg$ is defined, then $fg=gf$,
\item[(II)] if $fg$, $(fg)h$ and $gh$ are defined, then $(fg)h=f(gh)$,
\item[(III)] if $fh$ and $gh$ are defined and $\alpha,\beta\in\C$, then $(\alpha f+\beta g)h=\alpha(fh)+\beta(gh)$,
\item[(IV)] there exists $e\in \frak H$ such that $e f=f$ for every $f\in \frak H$,
\item[(V)] if $f_n g$ are defined and $f_n\to f$, $f_ng\to h$, then $fg=h$.
\item[(VI)] The set $\frak M=\{f\in\frak H:\: \mbox{$fg$ is defined for all $g\in\frak H$} \}$ is dense in $\frak H$; moreover if $fg$ is defined, then there exist $f_n\in\frak M$ such that $f_n\to f$ and $f_n g\to fg$, 
\item[(VII)] for every $f\in\frak H$, there exists $\bar f\in\frak H$ such that $\langle fg,h\rangle=\langle g,\bar f h\rangle$ for all $f,g\in\frak M$.
\end{itemize}
\end{defin}
An important result by Rokhlin is that every unitary ring  can be written as $\frak H=L^2(M,\mathcal M,m)$, where $(M,\mathcal M, m)$  is a Lebesgue space (see, e.g., V.A. Rokhlin\footnote{The notions of \emph{Lebesgue space} used here allows point with positive measure, contrary to the classical case discussed by Ya.G. Sinai \cite{Sinai-topics-in-et} in the context of K-systems} \cite{Rohlin-1947}).
In our case we have the unitary ring $\mathcal H=L^2(X,\mathcal B,\Pi)$ and the subspace $H$ which is a sub-ring because of Theorem \ref{product-eigenfunctions}.
In this representation a subring $\frak R\subset\frak H$ corresponds to a $\sigma$-subalgebra $\mathcal N$ of $\mathcal M$, i.e. $\frak R=L^2(M,\mathcal N, m|_{\mathcal N})$. Therefore $H$ is a subspace of $\mathcal H$, which is a Hilbert space corresponding to some $\sigma$-subalgebra $\mathcal F$ of $\mathcal B$. Let us show that


\begin{prop}\label{H=mathcalH}
Up to sets of measure zero, 
$\mathcal F=\mathcal B$. In other words, $H=\mathcal H$.
\end{prop}
\begin{proof}
Let us use the technique of measurable partitions by Rokhlin (see \cite{Rohlin-1949}). According to it $\mathcal F$ corresponds to some measurable partition $\xi$ of $X$. If $\mathcal F\subsetneq\mathcal B$, then there exists a bounded, non-negative function $h(x)$ and a subset $A\in\mathcal F$ such that $\mathbb E(h|C_\xi)\geq \alpha$ for almost all $C_\xi\in A$ and some positive $\alpha$.
As any measurable function $h$ can be approximated arbitrarily well in $L^\infty(X,\mathcal F,\Pi|_{\mathcal F})$ sense by a function $h'$ which is a polinomial in the $x(s)$'s. Using (\ref{quality-approximation-xD(s)}) we can approximate $h'$ in the measure sense by a finite polynomial in the  eigenfunctions $\theta_\lambda$.
However, every such polynomial belongs to our Hilbert space $L^2(X,\mathcal F,\Pi|_{\mathcal F})$ and it is measurable with respect to $\mathcal F$. Therefore the conditional expectation of $h'$ with respect to $\xi$ is arbitrary close to $h'$, but such  a function cannot approximate $h$ in measure. This shows that $H=\mathcal H$.
\end{proof}
Propositions \ref{H=mathcalH} and \ref{eigenfuctions-basis-H} immediately give the following
\begin{cor}\label{cor-eigenfunctions-basis-mathcalH}
The set of eigenfunctions $\{\theta_\lambda\}_{\lambda\in\Lambda}$ is a basis in the space $\mathcal H$. 
\end{cor}
This fact, together with Theorem \ref{existence-of-eigenfunctions} and Corollary \ref{cor-eigenfunctions-nontrivial} yields Theorem \ref{thm-spectrum-of-U-is-Lambda}. It also implies the following 
\begin{theorem}\label{thm-ergodicity}
The dynamical system $(X,\mathcal B,\Pi,T)$ is ergodic.
\end{theorem}
\begin{proof}
By shift-invariance of $\Pi$ we already know that the eigenspace spanned by the constants is at least one-dimensional.
On the other hand, by Theorem \ref{thm-spectrum-of-U-is-Lambda}, its dimension cannot be bigger than one. 
This implies that the only invariant functions are  constants $\Pi$-almost everywhere, and hence we have ergodicity.
\end{proof}
Theorems \ref{thm-spectrum-of-U-is-Lambda} and \ref{thm-ergodicity} give part (i) or our Main Theorem.
\begin{remark}
One could also derive Corollary \ref{cor-eigenfunctions-basis-mathcalH} in a different way and without using the theory of unitary rings and measurable partitions by Rokhlin. The derivation, although explicit, is rather long. In fact, one can show that for every $-\infty<s_1<\ldots<s_r<\infty$ the product $x(s_1)\cdots x(s_r)$ belongs to the span of $\{\theta_\lambda\}_{\lambda\in\Lambda}$. For example, for $r=2$, by Theorems \ref{thm-innerproduct-theta_l,d^2-x(s)} and \ref{product-eigenfunctions},
\bey
x(s_1)x(s_2)&=&\left(\sum_{\lambda_1\in\Lambda}\left\langle x(s_1),\tilde\theta_{\lambda_1}\right\rangle\tilde\theta_{\lambda_1}\right)\left(\sum_{\lambda_2\in\Lambda}\left\langle x(s_2),\tilde\theta_{\lambda_2}\right\rangle\tilde\theta_{\lambda_2}\right)=\nonumber\\
&=&\sum_{\lambda_1,\lambda_2\in\Lambda}\lambda_1^{s_1}\lambda_2^{s_2} 
g(d_1) g(d_2)\mu(d)\tilde\theta_{\lambda_1\lambda_2}=\nonumber\\
&=&\sum_{\lambda\in\Lambda}\left(\mu(d)\sum_{\scriptsize{\ba{c}\lambda_1,\lambda_2\in\Lambda\\ \lambda_1\lambda_2=\lambda\ea}}
\lambda_1^{s_1}\lambda_2^{s_2}
g(d_1) g(d_2)\right)\tilde\theta_\lambda,\nonumber
\eey
and one can prove that
\be 
\sum_{\scriptsize{\ba{c}\lambda_1,\lambda_2\in\Lambda\\ \lambda_1\lambda_2=\lambda\ea}}
\left|g(d_1) g(d_2)\right|=\mathcal O_\varepsilon(d^{-2+\varepsilon}) 
\nonumber\ee
for every $\varepsilon>0$,
where $\lambda=e\!\left(\frac{l}{d^2}\right)$. This implies that 
\be\sum_{\lambda\in\Lambda}\left|\sum_{\scriptsize{\ba{c}\lambda_1,\lambda_2\in\Lambda\\ \lambda_1\lambda_2=\lambda\ea}}
\lambda_1^{s_1}\lambda_2^{s_2}
g(d_1) g(d_2) \right|^2\nonumber
\ee is finite.
\end{remark}

\section{Spectral analysis for $(\mathbb G,\mathbb B,\mathbb P,\mathbb T_{\mathbf u})$}\label{sec-spectrum-of-G}
Recall the group $\mathbb G$ defined in (\ref{groupG}).
Let us consider the space $\mathbb H=L^2(\mathbb G,\mathbb B, \mathbb P)$, and the unitary operator $\mathbb U$ on $\mathbb H$ defined by
\be(\mathbb U f)(\mathbf g)=f(\mathbf g+(1,1,1,\ldots)).\nonumber\ee
\begin{theorem}\label{spectrum-U-forG}
The spectrum of $\mathbb U$ is given by $\Lambda$.
\end{theorem}
\begin{proof}
Consider the projection $\pi_{p^2}:\mathbb G\to\mathbb Z/p^2\mathbb Z$, $\pi_{p^2}(\mathbf g)=g_{p^2}$. It is immediate to see that the function $\xi_{e(1/p^2)}(\mathbf g)=e\!\left(\pi_{p^2}(\mathbf g)/p^2\right)$ is an eigenfunction for $\mathbb U$ with eigenvalue $e\!\left(\tfrac{1}{p^2}\right)$. By taking powers one can get any eigenfunction $\xi_{e(t/p^2)}$ with any eigenvalue $e\!\left(\tfrac{t}{p^2}\right)$ for $0\leq t\leq p^2-1$. By multiplying different such eigenfunctions (with different $p$), one can obtain eigenfunctions $\xi_{\lambda}$ with an arbitrary eigenvalue $\lambda\in\Lambda$.
Since $\Lambda$ is the character group of $\mathbb G$ and $\mathbb T_{\textbf u}$ is a translation in $\mathbb G$, then  there are no other eigenvalues. 
\end{proof}

To conclude the proof of part (ii) of our Main Theorem we need the following
\begin{theorem}[J. von Neumann, \cite{vonNeumann-1932}]\label{thm-von-Neumann}
Two ergodic measure-preserving transformations with pure point spectrum are isomorphic if and only if they have the same spectrum.
\end{theorem}

Theorems \ref{spectrum-U-forG} and \ref{thm-von-Neumann} imply that $(X,\mathcal B,\Pi,T)$ and $(\mathbb G,\mathbb B,\mathbb P,\mathbb T_{\mathbf u})$ are isomorphic as measure-preserving dynamical systems. This concludes the proof of our Main Theorem.

\appendix
\section{The Proof of Theorem \ref{thm1''}}\label{appendix-A}
This Appendix is dedicated to the  proof of Theorem \ref{thm1''}. It is based on the following identity
\be\left(\sum_{n=1}^\infty\frac{a(n)}{n^s}\right)\left(\sum_{n=1}^\infty \frac{b(n)}{n^s}\right)=\sum_{n=1}^\infty\frac{(a*b)(n)}{n^s},\label{product-formula}\ee
where $a*b$ is the Dirichlet convolution of $a$ and $b$: \be (a*b)(n)=\sum_{d|n}a(d)b\!\left(\tfrac{n}{d}\right).\label{dirichlet-convolution}\ee
We shall be considering only the case of $s=2$ and bounded sequences $a(n)$ and $b(n)$, therefore there will be no question about convergence of the above series. 
We shall also use the classical identity \be\sum_{d|n}\mu(d)=0,\hspace{.5cm}\mbox{for $n>1$}.\label{identity-divisors}\ee

First, let us consider the case of square-free numbers not divisible by a single prime $p$, i.e. $\mathcal S=\{p\}$. In this case, we shall prove Theorem \ref{thm1''} by means of three lemmata, and then we shall explain how to generalize this approach to general finite sets $\mathcal S$.

Let $w_p(n)$ be the indicator of the integers not divisible by $p$, i.e.
\be w_p(n)=\begin{cases}0,&\mbox{if $p|n$;}\\1,&\mbox{otherwise}.\end{cases}\nonumber\ee

\begin{lem}\label{nth-term}
\be\mu^2(n)w_p(n)=\sum_{d:\:d^2|n}\mu(d)w_p(d)w_p\!\left(\tfrac{n}{d}\right)\label{identity-nth-term}\ee
\end{lem}

\begin{proof}
If $p|n$, then $p|d$ or $p|\tfrac{n}{d}$ (possibly both) for every divisor $d$ of $n$. Thus $w_p(n)w_p\!\left(\tfrac{n}{d}\right)=0$ and the sum in the rhs of (\ref{identity-nth-term}) is 0 (and obviously equals the lhs).  If $p\nmid n$, then no divisor $d$ of $n$ will be divisible by $p$ and $w_p(d)=w_p\!\left(\tfrac{n}{d}\right)$=1. The sum in (\ref{identity-nth-term}) then becomes $\sum_{d^2|n}\mu(d)$. If $n$ is square-free, then $d^2=1$ is the only perfect square that divides $n$, and the sum equals 1 (and clearly agrees with the lhs of (\ref{identity-nth-term})). If $n$ is not square-free let us write it as $n_1 n_2^2$ where $n_1$ and $n_2$ are defined as follows. For every prime $p$ let us define $l=l(p,n)=\max\{j:\: p^j|n\}$; then set $n_1=\prod_{p:\: 2\nmid l}p$ and $n_2=\prod_{p:\: 2|l}p^{\frac{l}{2}}\cdot\prod_{p:\: 2\nmid l}p^{\frac{l-1}{2}}$. Since $n_1$ is square-free, if $d^2|n$, then $d|n_2$. This means that the sum in the rhs of (\ref{identity-nth-term}) becomes $\sum_{d|n_2}\mu(d)$ and equals 0 by (\ref{identity-divisors}) (thus matching the lhs). This concludes the proof of the Lemma. 
\end{proof}

\begin{lem}\label{lemma-sum-w_p}
\be\sum_{n=1}^\infty \frac{w_p(n)}{n^2}=\frac{p^2-1}{p^2}\zeta(2).\nonumber\ee
\end{lem}
\begin{proof}
The formula follows from the trivial computation
\be\sum_{n=1}^\infty \frac{w_p(n)}{n^2}=\sum_{n=1}^\infty \frac{1}{n^2}-\sum_{n=1}^\infty\frac{1}{(pn)^2}=\left(1-\frac{1}{p^2}\right)\zeta(2).\nonumber\ee
\end{proof}
Let us denote by $\{\delta_1(n)\}_{n\geq 1}$ the sequence equal to $1$ if $n=1$ and 0 otherwise. Then
\begin{lem}\label{lemma-convolution}
\be(\mu w_p)*w_p=\delta_1.\nonumber\ee
\end{lem}
\begin{proof}
For $n=1$ the statement is obvious since $d=1$ is the only divisor of $n$ and we have $\mu(1)w_p(1)^2=1=\delta_1(1)$. Let $n>1$. Then $\left((\mu w_p)*w_p\right)(n)=\sum_{d|n}\mu(d)w_p(d)w_p\!\left(\tfrac{n}{d}\right)$. We can discuss the cases when $p|n$ and  $p\nmid n$ separately, and argue as in the proof of Lemma \ref{nth-term}. In the first case we have that 
$w_p(n)w_p\!\left(\tfrac{n}{d}\right)=0$ and the sum is 0. In the second case 
$w_p(d)=w_p\!\left(\tfrac{n}{d}\right)$=1 and the sum becomes $\sum_{d|n}\mu(d)$, that is 0 by (\ref{identity-divisors}). In other words, we have shown that, for $n>1$, we have$\left((\mu w_p)*w_p\right)(n)=0=\delta_1(n)$, and this concludes the proof of the Lemma.
\end{proof}
\begin{cor}\label{cor-series-mu-wp}
\be\sum_{n=1}^\infty\frac{\mu(n)w_p(n)}{n^2}=\frac{p^2}{p^2-1}\frac{1}{\zeta(2)}.\nonumber\ee
\end{cor}
\begin{proof}
This corollary is a straightforward application of Lemma \ref{lemma-sum-w_p} and the formul\ae\: (\ref{product-formula}, \ref{dirichlet-convolution}) with $a=\mu w_p$, $b=w_p$, and (from Lemma \ref{lemma-convolution}) $a*b=\delta_1$.
\end{proof}
We can now give the 
\begin{proof}[Proof of Theorem \ref{thm1''} when $\mathcal S=\{p\}$]
Notice that 
$|\mathcal Q_N^{\{p\}}(0)|=\sum_{n\leq N}\mu^2(n)w_p(n)$. By Lemma \ref{nth-term}, we can write
\be|\mathcal Q_N^{\{p\}}(0)|=\sum_{n\leq N}\mu^2(n)w_p(n)=\sum_{n\leq N}\sum_{d^2|n}\mu(d)w_p(d)w_p\!\left(\tfrac{n}{d}\right).\label{pf1}\ee
Now we want to exchange the two sums. Let us fix $d\leq \sqrt N$. For every $n$ of the form $n=m d^2$ we have $w_p\!\left(\tfrac{n}{d}\right)=w_p(m)w_p(d)$. Let $\eta_d^{\{p\}}(N)$ be the number of integers of the form $m d^2$ where $m\leq \tfrac{N}{d^2}$ and $p\nmid m$. Then
\be|\mathcal Q_N^{\{p\}}(0)|=\sum_{d\leq \sqrt N}\eta_d^{\{p\}}(N)\mu(d)w_p(d).\nonumber\ee
We can estimate the number $\eta_d^{\{p\}}(N)$ as follows. Let $\left\lfloor\tfrac{N}{d^2}\right\rfloor\equiv t \:\:(\!\!\!\!\mod p)$, $t\in\{0,1,2,\ldots, p-1\}$.
Then \be\eta_d^{\{p\}}(N)=\frac{\left\lfloor\tfrac{N}{d^2}\right\rfloor-t}{p}(p-1)+t=\frac{N}{d^2}\frac{p-1}{p}+q_1^{\{p\}}(d,N),\nonumber\ee
where
\be\left|q_1^{\{p\}}(d,N)\right|\leq\frac{p-1}{p}\left|\left\lfloor\frac{N}{d^2}\right\rfloor-\frac{N}{d^2}\right|+t\left(1-\frac{p-1}{p}\right)\leq2\, \frac{p-1}{p}=:C'(\{p\}).\nonumber\ee
This gives us
\be|\mathcal Q_N^{\{p\}}(0)|=N\frac{p-1}{p}\sum_{d\leq\sqrt N}\frac{\mu(d)w_p(d)}{d^2}+q_2^{\{p\}}(N),\nonumber\ee
where $|q_2^{\{p\}}(N)|\leq C'(\{p\})\sqrt N$. 
Now, Corollary \ref{cor-series-mu-wp} yields
\be|\mathcal Q_N^{\{p\}}(N)|=\frac{p}{p+1}\frac{1}{\zeta(2)}N+q_2^{\{p\}}(N)+q_3^{\{p\}}(N),\nonumber\ee
where \be |q_3^{\{p\}}(N)|\leq N\frac{p-1}{p}\sum_{d>\sqrt N}\frac{1}{d^2}\leq N\frac{p-1}{p}\int_{\sqrt N}^\infty\frac{\de x}{(x-1)^2}=\frac{p-1}{p}\frac{N}{\sqrt N-1}\leq C''(\{p\})\sqrt N\nonumber\ee
for every $N\geq4$, where $C''(\{p\})=2\frac{p-1}{p}$. This concludes the proof of the theorem, with $\alpha(\{p\})=\frac{p}{p+1}$ and $C(\{p\})=C'(\{p\})+C''(\{p\})$.
\end{proof}

Let us now discuss how to adapt the above proof of Theorem \ref{thm1''} for the case of a general finite set of primes $\mathcal S\subset\mathcal P$. 
The sequence $w_p$ has to be replaced by the indicator of the integers divisible by none of the primes in $\mathcal S$, i.e.
\be w_{\mathcal S}(n)=\begin{cases}0,&\mbox{if $p|n$ for some $p\in\mathcal S$;}\\1,&\mbox{otherwise}.\end{cases}\nonumber\ee 
Lemma \ref{nth-term} is still valid if we replace $w_p$ by $w_{\mathcal S}$:
\begin{lem}\label{nth-term1}
\be\mu^2(n)w_{\mathcal S}(n)=\sum_{d:\:d^2|n}\mu(d)w_{\mathcal S}(d)w_{\mathcal S}\!\left(\tfrac{n}{d}\right)\label{identity-nth-term1}\ee
\end{lem}
Lemma \ref{lemma-sum-w_p} is replaced by an analogous statement given by inclusion-exclusion:
\begin{lem}\label{lemma-sum-w_p1}
\be\sum_{n=1}^\infty \frac{w_{\mathcal S}(n)}{n^2}=a(\mathcal S)\zeta(2),\nonumber\ee
where \be a(\mathcal S)=\prod_{p\in\mathcal S}\frac{p^2-1}{p^2}.\nonumber\ee
\end{lem}
\begin{proof}
If $\mathcal S=\{p_1,p_2,\ldots,p_k\}$, then
inclusion-exclusion gives
\bey a(\mathcal S)&=&\left(1-\sum_{i=1}^k\frac{1}{p_i^2}+\sum_{1\leq i_1<i_2\leq k}\frac{1}{(p_{i_1}p_{i_2})^2}-\sum_{1\leq i_1<i_2<i_3\leq k}\frac{1}{(p_{i_1}p_{i_2}p_{i_3})^2}+\ldots+\frac{(-1)^k}{(p_1p_2\cdots p_k)^2}\right)=\nonumber\\
&=&\prod_{i=1}^k\left(1-\frac{1}{p_i^2}\right).\nonumber\eey
\end{proof}
Lemma \ref{lemma-convolution} also holds for $w_{\mathcal S}$:
\begin{lem}\label{lemma-convolution1}
\be(\mu w_{\mathcal S})*w_{\mathcal S}=\delta_1.\nonumber\ee
\end{lem}
Finally, Corollary \ref{cor-series-mu-wp} is replaced by
\begin{cor}\label{cor-series-mu-wp1}
\be\sum_{n=1}^\infty\frac{\mu(n)w_{\mathcal S}(n)}{n^2}=\frac{1}{a(\mathcal S)\zeta(2)} \nonumber\ee
\end{cor}
We are now ready to give the
\begin{proof}[Proof of Theorem \ref{thm1''} for general $\mathcal S=\{p_1,\ldots,p_k\}\subset\mathcal P$]
Lemma \ref{nth-term1} gives
\be|\mathcal Q_N^{\mathcal S}(0)|=\sum_{n\leq N}\sum_{d^2|n}\mu(d)w_{\mathcal S}(d)w_{\mathcal S}\!\left(\tfrac{n}{d}\right).\label{pf11}\ee
Let us fix $d\leq \sqrt N$. For every $n$ of the form $n=m d^2$ we have $w_{\mathcal S}\!\left(\tfrac{n}{d}\right)=w_{\mathcal S}(m)w_{\mathcal S}(d)$. Let $\eta_d^{\mathcal S}(N)$ be the number of integers of the form $m d^2$ where $m\leq \tfrac{N}{d^2}$ and $p\nmid m$ for every $p\in\mathcal S$. Then
\be|\mathcal Q_N^{\mathcal S}(0)|=\sum_{d\leq \sqrt N}\eta_d^{\mathcal S}(N)\mu(d)w_{\mathcal S}(d).\nonumber\ee
The set of numbers not divisible by any $p\in\mathcal S$, has density given by
\bey
&&1-\sum_{i=1}^k\frac{1}{p_i}+\sum_{1\leq i_1<i_2\leq k}\frac{1}{p_{i_1}p_{i_2}}-\sum_{1\leq i_1<i_2<i_3\leq k}\frac{1}{p_{i_1}p_{i_2}p_{i_3}}+\ldots+\frac{(-1)^k}{p_1p_2\cdots p_k}=\nonumber\\
&&=\prod_{i=1}^k\left(1-\frac{1}{p_i}\right) 
=\prod_{p\in\mathcal S}\frac{p-1}{p}\nonumber\eey
The estimate of $\eta_d^{\mathcal S}(N)$ comes from the following observation. If 
\be \left\lfloor\frac{N}{d^2}\right\rfloor\equiv t\:\:\left(\!\!\!\!\!\!\mod [\mathcal S] 
\right)\hspace{.5cm}\mbox{for $t\in\left\{0,1,2,\ldots,[\mathcal S]
-1\right\}$},\nonumber\ee
then
\be\eta_d^{\mathcal S}(N)=\frac{\left\lfloor\frac{N}{d^2}\right\rfloor-t}{[\mathcal S]
}\prod_{p\in\mathcal S} (p-1)+t=\frac{N}{d^2}\prod_{p\in\mathcal S}\frac{p-1}{p}+q_1^{\mathcal S}(d,N),\nonumber\ee
where 
\bey
|q_1^{\mathcal S}(d,N)|&\leq& \prod_{p\in\mathcal S}\frac{p-1}{p}\left(\left\lfloor\frac{N}{d^2}\right\rfloor-\frac{N}{d^2}\right)+t\!\left(1-\prod_{p\in\mathcal S}\frac{p-1}{p}\right)\leq\nonumber\\
&\leq& 2\prod_{p\in\mathcal S}\frac{p-1}{p}+\left(\prod_{p\in\mathcal S} p -1\right)-\prod_{p\in\mathcal S}(p-1)=:C'(\mathcal S).\nonumber\eey
This gives
\be|\mathcal Q_N^{\mathcal S}(0)|=N\prod_{p\in\mathcal S}\frac{p-1}{p}\sum_{d\leq\sqrt N}\frac{\mu(d)w_{\mathcal S}(d)}{d^2}+q_2^{\mathcal S}(N),\nonumber\ee
where $|q_2^{\mathcal S}(N)|\leq C'(\mathcal S)\sqrt N$. Now Corollary \ref{cor-series-mu-wp1} yields
\be
|\mathcal Q_N^{\mathcal S}(0)|=\prod_{p\in\mathcal S}\frac{p}{p+1}\frac{1}{\zeta(2)}N+q_2^{\mathcal S}(N)+q_3^{\mathcal S}(N),\nonumber
\ee
where
\be |q_3^{\mathcal S}(N)|\leq N\prod_{p\in\mathcal S} \frac{p-1}{p}\sum_{d>\sqrt N}\frac{1}{d^2}\leq C''(\mathcal S)\sqrt N\nonumber\ee
and $C''(\mathcal S)=2\prod_{p\in\mathcal S} \frac{p-1}{p}$ for $N\geq4$. This concludes the proof of the general case of the theorem, with $\alpha(\mathcal S)=\prod_{p\in\mathcal S}\frac{p}{p+1}$ and $C(\mathcal S)=C'(\mathcal S)+C''(\mathcal S)$.
\end{proof}

\section{The Proofs of Lemmata \ref{lem-sum-c_2(d^2l)}-
\ref{lem-sum-c_2(d^2l+t)-general}}\label{appendix-B}
\begin{proof}[Proof of Lemma \ref{lem-sum-c_2(d^2l)}]
Let us write $l$ as follows:\be l=\prod_{\bar p\in \bar{\mathcal P}} (\bar{p})^{a(\bar{p})}\cdot\prod_{p'\in\mathcal P'}(p')^{b(p')}\cdot q\label{write-l}\ee
where $\bar{\mathcal P}\subseteq\mathcal P(d)$, $a(\bar p)\geq 2$ for every $\bar p\in\bar{\mathcal P}$, $\mathcal P'\subset\mathcal P\smallsetminus\mathcal P(d)$, $|\mathcal P|<\infty$, $b(p')\geq 2$ for every $p'\in\mathcal P'$, $q$ is square-free and $p\nmid q$ for every $p\in\bar{\mathcal P}\cup\mathcal P'$. It is clear that every $l\geq 1$ can be written uniquely as in (\ref{write-l}). And the condition $l\leq N$ can be rewritten using the notation in (\ref{def-Q_N^S(0)}) as 
\be q\in\mathcal Q^{\bar{\mathcal P}\cup \mathcal P'}_{N/\left(\prod_{\bar p\in\bar{\mathcal P}}(\bar p)^{a(\bar p)}\cdot\prod_{p'\in\mathcal P'}(p')^{b(p')}\right)}(0).\nonumber\ee
Furthermore, notice that $c_2(d^2 l)$ can depends only on $d$ and $\mathcal P'$:
\be c_2(d^2 l)=\prod_{p\in P(d)\cup\mathcal P'}\left(1-\frac{1}{p^2}\right)\prod_{p\notin \mathcal P(d)\cup\mathcal P'}\left(1-\frac{2}{p^2}\right) =\sigma_1\cdot\prod_{p\in\mathcal P(d)\cup\mathcal P'}\frac{p^2-1}{p^2-2}\label{c_2(d^l)}.\nonumber\ee
Now we can write
\bey
&&\frac{1}{N}\sum_{l\leq N}c_2(d^2 l)=\sigma_1\prod_{p\in \mathcal P(d)}\frac{p^2-1}{p^2-2}\sum_{\bar{\mathcal P}\subseteq \mathcal P(d)}\sum_{\mathcal P'\subset\mathcal P\smallsetminus\mathcal P(d)}\prod_{p'\in \mathcal P'}\frac{p'^2-1}{p'^2-2}\nonumber\\
&&\hspace{.3cm}\sum_{\scriptsize{\ba{c}a(\bar p)\geq 2,\:\bar p\in\bar{\mathcal P}\\b(p')\geq 2,\:p'\in\mathcal P'\ea}}\frac{1}{\prod_{\bar p\in\bar{\mathcal P}}(\bar p)^{a(\bar p)}\cdot\prod_{p'\in\mathcal P'}(p')^{b(p')}}\mathbb E^{\bar{\mathcal P}\cup\mathcal P'}_{N/\left(\prod_{\bar p\in\bar{\mathcal P}}(\bar p)^{a(\bar p)}\cdot\prod_{p'\in\mathcal P'}(p')^{b(p')}\right)}(0)\nonumber
\eey
Now we can use (\ref{densities-E_N^S(0)}) while taking the limit as $N\to\infty$, and sum over all $a(\bar p), b(p')\geq 2$ as in the proof of Proposition \ref{thm-density-level-sets-for-c2}. 
Notice that the sets $\mathcal P(d)$ and $\mathcal P'$ are disjoint. We obtain
\bey
\lim_{N\to\infty}\frac{1}{N}\sum_{l\leq N}c_2(d^2 l)&=&\sigma_1\prod_{p\in\mathcal P(d)}\frac{p^2-1}{p^2-2}\sum_{\bar{\mathcal P}\subseteq \mathcal P(d)}\sum_{\mathcal P'\subset\mathcal P\smallsetminus\mathcal P(d)}\prod_{p'\in \mathcal P'}\frac{p'^2-1}{p'^2-2}\frac{6}{\pi^2}\prod_{p\in\bar{\mathcal P}\cup\mathcal P'}\frac{1}{p^2-1}=\nonumber\\
&=&\sigma_1\frac{6}{\pi^2}\prod_{p\in\mathcal P(d)}\frac{p^2-1}{p^2-2}\left(1+\frac{1}{p^2-1}\right)\prod_{p\in\mathcal P\smallsetminus\mathcal P(d)}\left(1+\frac{1}{p^2-2}\right)=\nonumber\\
&=&\frac{6}{\pi^2}\prod_{p\in\mathcal P(d)}\frac{p^2}{p^2-1}\prod_{p}\frac{p^2-1}{p^2}=\left(\frac{6}{\pi^2}\right)^2\prod_{p\in\mathcal P(d)}\frac{p^2}{p^2-1},\nonumber\eey
and the lemma is thus proved.
\end{proof}
\begin{proof}[Proof of Lemma \ref{lem-sum_c_2(d^l+t)-square-free}] 
Let us first consider the case $g=1$. Numbers of the form
$\prod_{p'\in\mathcal P'}p'^{b(p')} \cdot q$, 
where $\mathcal P'\subset\mathcal P\smallsetminus\mathcal P(d)$, $|\mathcal P'|<\infty$, $b(p')\geq 2$ for $p'\in\mathcal P'$, $q$ is square-free and $p\nmid q$ for every $p\in\mathcal P(d)\cup\mathcal P'$
can be represented as
\be \prod_{p'\in\mathcal P'}p'^{b(p')} \cdot q=d^2 l+h\ee
for some $1\leq h\leq d^2-1$, where $\gcd(h,d^2)=1$. Since there are $\varphi(d^2)$ such $h$'s (here $\varphi$ denotes Euler's totient function) and the various $h$'s appear with the same frequency, then
\be\lim_{N\to\infty}\frac{1}{N}\sum_{l\leq N}c_2(d^2 l+t)=\frac{1}{\varphi(d^2)}\lim_{N\to\infty}\frac{1}{N}\sum_{l\leq N}\sum_{\gcd(h,d^2)=1}c_2(d^2 l+h)\label{intermediate-step-1/varphi}\ee
Notice that the condition $l\leq N$ becomes \be q\in\mathcal Q^{\mathcal P(d)\cup\mathcal P'}_{(d^2N+h)/\left(\prod_{p'\in\mathcal P'}p'^{b(p')}\right)}(0)\nonumber\ee
and
\be c_2(d^2l+t)=\prod_{p'\in\mathcal P'}\left(1-\frac{1}{p^2}\right)\prod_{p\notin \mathcal P'}\left(1-\frac{2}{p^2}\right)=\sigma_1\cdot\prod_{p'\in\mathcal P'}\frac{p'^2-1}{p'^2-2}.\nonumber\ee
Now
\bey &&\frac{1}{N}\sum_{l\leq N}\sum_{\gcd(h,d^2)=1}c_2(d^2 l+h)=\sigma_1\sum_{\mathcal P'\subset\mathcal P\smallsetminus \mathcal P(d)}\prod_{p'\in\mathcal P'}\frac{p'^2-1}{p'^2-2}\nonumber\\
&&\hspace{.3cm}\sum_{b(p')\geq2,\:p'\in\mathcal P'}\frac{d^2}{\left(\prod_{p'\in\mathcal P'}p'^{b(p')}\right)}\frac{d^2 N+h}{d^2 N}\mathbb E^{\mathcal P(d)\cup\mathcal P'}_{(d^2N+h)/\left(\prod_{p'\in\mathcal P'}p'^{b(p')}\right)}(0),\nonumber\eey
and by taking the limit as $N\to\infty$ we get
\bey
&&\lim_{N\to\infty}\frac{1}{N}\sum_{l\leq N}\sum_{\gcd(h,d^2)=1}c_2(d^2 l+h)=\sigma_1\sum_{\mathcal P'\subset\mathcal P\smallsetminus\mathcal P(d)}\prod_{p'\in\mathcal P'}\frac{p'^2-1}{p'^2-2}\frac{6}{\pi^2}\prod_{p\in\mathcal P(d)}\frac{p^3}{p+1}\prod_{p'\in\mathcal P'}\frac{1}{p'^2-1}=\nonumber\\
&&=\sigma_1\frac{6}{\pi^2}\prod_{p\in\mathcal P(d)}\frac{p^3}{p+1}\sum_{\mathcal P'\subset\mathcal P\smallsetminus\mathcal P(d)}\prod_{p'\in\mathcal P'}\frac{1}{p'^2-2}=\sigma_1\frac{6}{\pi^2}\prod_{p\in\mathcal P(d)}\frac{p^3}{p+1}\prod_{p\in\mathcal P\smallsetminus\mathcal P(d)}\left(1+\frac{1}{p^2-2}\right)=\nonumber\\
&&=\frac{6}{\pi^2}\prod_{p\in\mathcal P(d)}\frac{p^3(p^2-2)}{(p+1)(p^2-1)}\prod_{p}\frac{p^2-1}{p^2}=\left(\frac{6}{\pi^2}\right)^2\prod_{p\in\mathcal P(d)}\frac{p^3(p^2-2)}{(p+1)(p^2-1)}\label{pf-lem-sum-c_2(d^2l+1)}.
\eey
Let us apply the fact that $\varphi$ is multiplicative and that $\varphi(p^2)=p(p-1)$. We obtain
\be\frac{1}{\varphi(d^2)}=
\prod_{p\in\mathcal P(d)}\frac{1}{p(p-1)}\label{pf-lem-sum-c_2(d^2l+1)-1}.\ee
Now (\ref{intermediate-step-1/varphi}), (\ref{pf-lem-sum-c_2(d^2l+1)}) and (\ref{pf-lem-sum-c_2(d^2l+1)-1}) yield the desired result.
Let us now consider the 
case when $\gcd(t,d^2)=g=\bar p$. 
In this case $d^2=\bar p^2 d_1^2$ and $t=\bar p t_1$, where $d_1$ is square-free, $\bar p\nmid d_1$, and $\bar p\nmid t_1$. We can write
 \be \frac{d^2}{\bar p}l+\frac{t}{\bar p}=\bar p d_1^2 l+t_1=\prod_{p'\in\mathcal P'}p'^{b(p')}\cdot q_1,\nonumber\ee
where $\mathcal P'\subset\mathcal P\smallsetminus\mathcal P(d)$, $|\mathcal P'|<\infty$, $q_1$ is square-free and $p\nmid q$ for every $p\in \mathcal P(d)\cup\mathcal P'$. The condition $l\leq N$ reads now as
\be q_1\in\mathcal Q^{\mathcal P(d)\cup\mathcal P'}_{(\bar p d_1^2 N+t_1)/\left(\prod_{p'\in\mathcal P'}p'^{b(p')}\right)}(0).\nonumber\ee
Since, by assumption, $\bar p^2\nmid d^2 l+t$, then we have
\be c_2(d^2 l+t)=c_2(\bar pd_1^2+t_1)=\prod_{p'\in\mathcal P'}\left(1-\frac{1}{p'^2}\right)\prod_{p\notin\mathcal P'}\left(1-\frac{2}{p^2}\right)=\sigma_1\cdot\prod_{p'\in\mathcal P'}\frac{p'^2-1}{p'^2-2}\nonumber\ee
Now, since $\gcd(t_1,\bar p d_1^2)=1$, we can use (\ref{intermediate-step-1/varphi}):
\be
\lim_{N\to\infty}\frac{1}{N}\sum_{l\leq N}c_2(d^2 l+t)=\lim_{N\to\infty}\frac{1}{N}\sum_{l\leq N} c_2(\bar p d_1^2 l+t_1)=\frac{1}{\varphi(\bar p d_1^2)}\lim_{N\to\infty}\frac{1}{N}\sum_{l\leq N}\sum_{\gcd(h_1,\bar p d_1^2)=1}c_2(\bar p d_1^2+h_1),\label{intermediate-step-1/varphi-2b}
\ee
and we can write
\bey
&&\frac{1}{N}\sum_{l\leq N}\sum_{\gcd(h_1,\bar p d_1^2)=1}c_2(\bar pd_1^2+h_1)=\sigma_1\sum_{\mathcal P'\subset\mathcal P\smallsetminus \mathcal P(d)}\prod_{p'\in\mathcal P'}\frac{p'^2-1}{p'^2-2}\nonumber\\
&&\hspace{.3cm}\sum_{b(p')\geq 2,\: p'\in\mathcal P'}\frac{\bar p d_1^2}{\prod_{p'\in\mathcal P'} p'^{b(p')}}\frac{\bar p d_1^2 N+h_1}{\bar p d_1^2 N}\mathbb E^{\mathcal P(d)\cup \mathcal P'}_{(\bar p d_1^2 N+h_1)/\left(\prod_{p'\in\mathcal P'}p'^{b(p')}\right)}(0).\nonumber\eey
Notice that $\mathcal P(d)=\mathcal P(d_1)\cup\{\bar p\}$. By taking the limit as $N\to\infty$ we obtain
\bey
&&\lim_{N\to\infty}\frac{1}{N}\sum_{l\leq N}\sum_{\gcd(h_1,\bar p d_1^2)=1}c_2(\bar p d_1^2l+h_1)=\nonumber\\
&&=\sigma_1\sum_{\mathcal P'\subset\mathcal P\smallsetminus \mathcal P(d)}\prod_{p'\in\mathcal P'}\frac{p'^2-1}{p'^2-2}\bar p d_1^2\frac{6}{\pi^2}\prod_{p\in\mathcal P(d)}\frac{p}{p+1}\prod_{p'\in\mathcal P'}\frac{1}{p'^2-1}=\nonumber\\
&&=\sigma_1\frac{6}{\pi^2}\frac{\bar p^2}{\bar p+1}\prod_{p\in\mathcal P(d_1)}\frac{p^3}{p+1}\sum_{\mathcal P'\subset\mathcal P\smallsetminus\mathcal P(d)}\prod_{p'\in\mathcal P'}\frac{1}{p'^2-2}=\nonumber\\
&&=\sigma_1\frac{6}{\pi^2}\frac{\bar p^2}{\bar p+1}\prod_{p\in\mathcal P(d_1)}\frac{p^3}{p+1}\prod_{p\in\mathcal P\smallsetminus \mathcal P(d)}\left(1+\frac{1}{p^2-2}\right)=\nonumber\\
&&=\frac{6}{\pi^2}\frac{\bar p^2}{\bar p+1}\frac{\bar p^2-2}{\bar p^2-1}\prod_{p\in\mathcal P(d_1)}\frac{p^3(p^2-2)}{(p+1)(p^2-1)}\prod_{p}\frac{p^2-1}{p^2}=\nonumber\\
&&=\left(\frac{6}{\pi^2}\right)^2\frac{\bar p^2(\bar p^2-2)}{(\bar p+1)(\bar p^2-1)}\prod_{p\in\mathcal P(d_1)}\frac{p^3(p^2-2)}{(p+1)(p^2-1)}.\label{intermediate-step-1/varphi-3b}
\eey
Let us use the fact that $\varphi(\bar p d_1^2)=\varphi(\bar p)\varphi(d_1^2)=(\bar p-1)\varphi(d_1^2)$ to obtain 
the formula 
\be\frac{1}{\varphi(\bar pd_1^2)}
=\frac{1}{\bar p-1}\prod_{p\in\mathcal P(d_1)}\frac{1}{p(p-1)}.\label{intermediate-step-1/varphi-4b}\ee
Now we can combine (\ref{intermediate-step-1/varphi-2b}), (\ref{intermediate-step-1/varphi-3b}), and (\ref{intermediate-step-1/varphi-4b}) to conclude the proof of the Lemma.
The case of a general square-free $g$ is treated analogously.
\end{proof}
\begin{proof}[Proof of Lemma \ref{lem-sum-c_2(d^2l+t)-general}]
The case when $g$ is square-free (i.e. $\mathcal P_2(g)=\varnothing$) is already included in Lemma \ref{lem-sum_c_2(d^l+t)-square-free}. Thus, it is enough to consider the case when $\mathcal P_2(g)\neq\varnothing$.
Let, for simplicity, $\gcd(t,d^2)=g=\bar p^2$ (i.e. $\omega(g)=1)$, the case of $\omega(g)>1$ being analogous. We have that $d^2=\bar p^2 d_1^2$ and $t=\bar p^2 t_1$, where $d_1$ is square-free and $\bar p\nmid d_1$. 
In particular $\gcd(t_1,d_1^2)=1$.
We can write
\be\frac{d^2}{\bar p^2}l+\frac{t}{\bar p^2}=d_1^2 l+t_1=\bar p^{a}\prod_{p\in\mathcal P'}p'^{b(p')}q_1,\nonumber\ee
where $a\geq 0$, $\mathcal P'\subset \mathcal P\smallsetminus\mathcal P(d)$, $|\mathcal P'|<\infty$, $q_1$ is square-free and $p\nmid q_1$ for every $p\in\mathcal P(d)\cup\mathcal P'$. The condition $l\leq N$ can be written as
\be q_1\in\mathcal Q^{\mathcal P(d)\cup \mathcal P'}_{(d_1^2 N+t_1)/\left(\prod_{p'\in\mathcal P'}p'^{b(p')}\right)}(0),\nonumber\ee
and 
\be c_2(d^2 l+t)=c_2(\bar p^2(d_1^2l+t_1))=\prod_{p\in\mathcal P'\cup\{\bar p\}}\left(1-\frac{1}{p^2}\right)\prod_{p\notin \mathcal P'\cup \{\bar p\}}\left(1-\frac{2}{p^2}\right)=\sigma_1\cdot\prod_{p\in\mathcal P'\cup\{\bar p\}}\frac{p^2-1}{p^2-2}.\nonumber\ee
Notice that by $\mathcal P'$ and $\{\bar p\}$ are disjoint by construction. Using (\ref{intermediate-step-1/varphi}) we see that
\bey
&&\lim_{N\to\infty}\frac{1}{N}\sum_{l\leq N}c_2(d^2l+t)=\lim_{N\to\infty}\sum_{l\leq N}c_2(\bar p^2(d_1^2l+t_1))=\nonumber\\
&&=\frac{1}{\varphi(d_1^2)}\lim_{N\to\infty}\frac{1}{N}\sum_{l\leq N}\sum_{\gcd(h_1,d_1^2)=1}c_2(\bar p^2(d_1^2l+h_1)).\label{lim-varphi(d_1^2)}
\eey
We have
\bey
&&\frac{1}{N}\sum_{l\leq N}\sum_{\gcd(h_1,d_1^2)=1}c_2(\bar p^2(d_1^2l+h_1))=\sigma_1\sum_{\mathcal P'\subset \mathcal P\smallsetminus\mathcal P(d)}\frac{\bar p^2-1}{\bar p^2-2}\prod_{p'\in\mathcal P'}\frac{p'^2-1}{p'^2-2}\nonumber\\
&&\hspace{.3cm}\sum_{a\geq0}\sum_{b(p')\geq2,\:p'\in\mathcal P'}\frac{d_1^2}{\bar p^a\prod_{p'\in\mathcal P'}p'^{b(p')}}\frac{d_1^2 N+h_1}{d_1^2 N}\mathbb E^{\mathcal P(d)\cup\mathcal P'}_{(d_1^2 N+h_1)/\left(\bar p^a\prod_{p'\in\mathcal P'}p'^{b(p')}\right)}(0),\nonumber
\eey
and by taking the limit as $N\to\infty$ we get
\bey
&&\lim_{N\to\infty}\frac{1}{N}\sum_{l\leq N}\sum_{\gcd(h_1,d_1^2)=1}c_2(\bar p^2(d_1^2 l+h_1))=\nonumber\\
&&=\sigma_1\frac{\bar p^2-1}{\bar p^2-2}\sum_{\mathcal P'\subset\mathcal P\smallsetminus \mathcal P(d)}\prod_{p'\in\mathcal P'}\frac{p'^2-1}{p'^2-2}\frac{d_1^2\bar p}{\bar p-1} \frac{6}{\pi^2}\prod_{p\in\mathcal P(d)}\frac{p}{p+1}\prod_{p'\in\mathcal P'}\frac{1}{p'^2-1}=\nonumber\\
&&=\sigma_1\frac{6}{\pi^2}\frac{\bar p+1}{\bar p(\bar p^2-2)}\prod_{p\in\mathcal P(d)}\frac{p^3}{p+1}\sum_{\mathcal P'\subset\mathcal P\smallsetminus\mathcal P(d)}\prod_{p'\in\mathcal P'}\frac{1}{p'^2-2}=\nonumber\\
&&=\sigma_1\frac{6}{\pi^2}\frac{\bar p+1}{\bar p(\bar p^2-2)}\prod_{p\in\mathcal P(d)}\frac{p^3}{p+1}\prod_{p\in \mathcal P\smallsetminus \mathcal P(d)}\left(1+\frac{1}{p^2-2}\right)=\nonumber\\
&&=\frac{6}{\pi^2}\frac{\bar p^2}{\bar p^2-1}\prod_{p\in\mathcal P(d_1)}\frac{p^3}{p+1}\frac{p^2-2}{p^2-1}\prod_p\frac{p^2-1}{p^2}=\left(\frac{6}{\pi^2}\right)^2\frac{\bar p^2}{\bar p^2-1}\prod_{p\in\mathcal P(d_1)}\frac{p^3(p^2-2)}{(p+1)(p^2-1)}.\label{lim-varphi(d_1^2)-1}
\eey
We use again the fact that 
\be\frac{1}{\varphi(d_1^2)}=
\prod_{p\in\mathcal P(d_1)}\frac{1}{p(p-1)}\label{lim-varphi(d_1^2)-2},\ee
and combining (\ref{lim-varphi(d_1^2)}), (\ref{lim-varphi(d_1^2)-1}) and (\ref{lim-varphi(d_1^2)-2}), we obtain the Lemma.
\end{proof}

\addcontentsline{toc}{section}{Bibliography}
\bibliographystyle{plain}
\bibliography{square-free-bibliography}
\end{document}